\documentclass[a4paper,12pt]{amsart}

\usepackage{verbatim}                          
\usepackage{amsmath,amsthm}                                        
\usepackage{color,textcomp,amssymb,german}
\usepackage{graphicx}
\usepackage[british]{babel}
\usepackage[all,import]{xy}
\usepackage[T1]{fontenc}                                    
\usepackage{times}                                       
\usepackage[mathscr]{eucal}
\swapnumbers
\newtheoremstyle{theorem}{8pt\vfill}{8pt\vfill}{\itshape}{}{\bfseries}{.}{.5em}{}
\newtheoremstyle{paragraph}{8pt\vfill}{8pt\vfill}{}{}{\bfseries}{}{.5em}{}
\theoremstyle{theorem}                                       
\newtheorem{thm}{Theorem}[section]
\newtheorem{prop}[thm]{Proposition}
\newtheorem{cor}[thm]{Corollary}
\newtheorem{lemma}[thm]{Lemma}
\theoremstyle{definition}
\newtheorem{df}[thm]{Definition}
\newtheorem{eg}[thm]{Example}
\newtheorem{rem}[thm]{Remark}
\theoremstyle{paragraph}                                     
\newtheorem{pg}[thm]{\!\!}
 
\newcommand{\norm}[1]{\left| #1 \right|}
\newcommand{\lrquot}[3]{#1 \,\backslash\, #2 \,/\, #3 }
\newcommand{\lquot}[2]{#1  \,\backslash\, #2 }
\newcommand{\rquot}[2]{#1  \, / \, #2 }
\newcommand{\putunder}[2]{\underset{#2}{#1}}

\def \smallmat #1 #2 #3 #4 {{\scriptstyle \begin{pmatrix} {#1} & {#2} \\ {#3} & {#4} \end{pmatrix}}}
\def \tinymat #1 #2 #3 #4 {\bigl( \begin{smallmatrix} {#1} & {#2} \\ {#3} & {#4} \end{smallmatrix} \bigr)}
\def\eqref #1{\textup{(}\ref{#1}\textup{)}}

\newcommand{\ses}[3]{\xymatrix{0\ar[r]&{#1}\ar[r]&{#2}\ar[r]&{#3}\ar[r]&0}}                

\def\blanc{\mspace{4mu}\cdot\mspace{4mu}}
\def\cT{\mathcal{T}}
\def\cO{\mathcal{O}}
\def\cQ{\mathcal{Q}}
\def\cV{\mathcal{V}}
\def\cA{\mathcal{A}}
\def\cH{\mathcal{H}}
\def\cF{\mathcal{F}}
\def\cL{\mathcal{L}}
\def\cI{\mathcal{I}}
\def\cJ{\mathcal{J}}
\def\cM{\mathcal{M}}
\def\cK{\mathcal{K}}
\def\cU{\mathcal{U}}
\def\cN{\mathcal{N}}
\def\cG{\mathcal{G}}
\def\cC{\mathcal{C}}
\def\cE{\mathcal{E}}
\def\cR{\mathcal{R}}

\def\ZZ{\mathbb{Z}}

\def\CC{\mathbb{C}}
\def\FF{\mathbb{F}}
\def\AA{\mathbb{A}}
\def\PP{\mathbb{P}}
\def\varE{\tilde E}
\DeclareMathOperator{\Hom}{Hom}
\DeclareMathOperator{\End}{End}
\DeclareMathOperator{\Aut}{Aut}
\DeclareMathOperator{\Ext}{Ext}
\DeclareMathOperator{\Gal}{Gal}
\DeclareMathOperator{\ch}{char}
\DeclareMathOperator{\GL}{GL}
\DeclareMathOperator{\PGL}{PGL}
\DeclareMathOperator{\SL}{SL}
\DeclareMathOperator{\Bun}{Bun}
\DeclareMathOperator{\Pic}{Pic}
\DeclareMathOperator{\Cl}{Cl}
\DeclareMathOperator{\PCl}{Cl^{pr}}
\DeclareMathOperator{\ECl}{Cl^{eff}}
\DeclareMathOperator{\Vertex}{Vert\,}
\DeclareMathOperator{\Edge}{Edge\,}
\DeclareMathOperator{\Stab}{Stab\,}
\DeclareMathOperator{\vol}{vol}
\DeclareMathOperator{\inv}{inv}
\def\PBun{\PP\!\Bun}
\def\PBundec{\PP\!\Bun_2^{\rm dec}}
\def\PBunindec{\PP\!\Bun_2^{\rm indec}}
\def\PBuntr{\PP\!\Bun_2^{\rm tr}}
\def\PBungi{\PP\!\Bun_2^{\rm gi}}
\def\tor{{\rm tor}}
\def\nsimeq{\simeq\hspace{-10.5pt}/\hspace{5.5pt}}
\def\smallnsimeq{\simeq\hspace{-5pt}/\hspace{2pt}}
\def\longhookrightarrow{\lhook\joinrel\longrightarrow}

\renewcommand{\date}[1]{\gdef\DD{#1}}\newcommand{\DD}{}

\title{Graphs of Hecke operators}
\author{Oliver Lorscheid}
\address{The City College of New York, Math. Dept., 160 Convent Ave., New York NY 10031, USA}
\email{olorscheid@ccny.cuny.edu}

\begin{document}

\begin{abstract}
 Let $X$ be a curve over $\FF_q$ with function field $F$. In this paper, we define a graph for each Hecke operator with fixed ramification. A priori, these graphs can be seen as a convenient language to organize formulas for the action of Hecke operators on automorphic forms. However, they will prove to be a powerful tool for explicit calculations and proofs of finite dimensionality results.

 We develop a structure theory for certain graphs $\cG_x$ of unramified Hecke operators, which is of a similar vein to Serre's theory of quotients of Bruhat Tits trees. To be precise, $\cG_x$ is locally a quotient of a Bruhat Tits tree and has finitely many components. An interpretation of $\cG_x$ in terms of rank $2$ bundles on $X$ and methods from reduction theory show that $\cG_x$ is the union of finitely many cusps, which are infinite subgraphs of a simple nature, and a nucleus, which is a finite subgraph that depends heavily on the arithmetics of $F$.

 We describe how one recovers unramified automorphic forms as functions on the graphs $\cG_x$. In the exemplary cases of the cuspidal and the toroidal condition, we show how a linear condition on functions on $\cG_x$ leads to a finite dimensionality result. In particular, we re-obtain the finite-dimensionality of the space of unramified cusp forms and the space of unramified toroidal automorphic forms.

 In an Appendix, we calculate a variety of examples of graphs over rational function fields.
\end{abstract}

\maketitle

\tableofcontents


\section*{Introduction}

\noindent
Hecke operators play a central r\^ole in the theory of automorphic forms. For classical modular forms, they are also computationally well understood. The theory of arithmetic quotients of the Bruhat-Tits tree as studied by Serre in \cite{Serre} allowed to study Hecke operator over $p$-adic fields by geometric methods. In this paper, we consider how to compute with Hecke operators for automorphic forms on $\PGL_2$ over a global function field. Our theory can be understood as a global counterpart to Serre's viewpoint over $p$-adic fields.


There are a few applications of Serre's theory to automorphic forms over global fields, which, however, mainly concentrate on rational function fields (cf.\ \cite{Gekeler1}, \cite{Gekeler2} and \cite{Gekeler-Nonnengardt}). The key ingredient of this application is the strong approximation property of $\SL_2$, as we will explain below. We begin with reminding the reader of the definition of a Bruhat-Tits tree. Though this paper is independent from Serre's book \cite{Serre}, we review some aspects of it since the global theory (as developed in this paper) and the local approach (as in Serre's book) go hand in hand. In later parts of the paper, we make a few remarks pointing out the connections with or the differences to Serre's theory.
 
Let $F$ be a global function field and $x$ be a fixed place. We denote by $F_x$ the completion of $F$ at $x$, by $\cO_x$ its integers, by $\pi_x\in\cO_x$ a uniformizer and by $q_x$ the cardinality of the residue field $\rquot{\cO_x}{(\pi_x)}\simeq\FF_{q_x}$. The Bruhat-Tits tree $\cT_x$ of $F_x$ is a graph with vertex set $\rquot{\PGL_2(F_x)}{\PGL_2(\cO_x)}$. There is an edge between two cosets $[g]$ and $[g']$ if and only if $[g']$ contains $g\tinymat 1 {} {} {\pi_x} $ or $g\tinymat {\pi_x} b {} 1 $ for some $b\in\FF_{q_x}$. Note that this condition is symmetric in $g$ and $g'$, so $\cT_x$ is a geometric graph. In fact, $\cT_x$ is a $(q_x+1)$-regular tree.

Every subgroup of $\PGL_2(F_x)$ acts on $\cT_x$ by multiplication from the left. We shall be interested in the following case. Let $\cO^x_F\subset F$ be the Dedekind ring of all elements $a\in F$ with $\norm a_y\leq 1$ for all places $y\neq x$. Put $\Gamma=\PGL_2(\cO^x_F)$. Serre investigates in \cite{Serre} the quotient graph $\lquot{\Gamma}{\cT_x}$. It is the union of a finite connected graph with a finite number of cusps. A cusp is an infinite graph of the form
\begin{center} \includegraphics{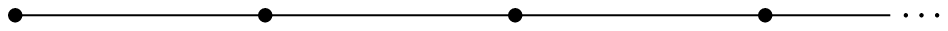}  \end{center}
and each cusp corresponds to an element of the class group of $\cO^x_F$.

An unramified automorphic form over $F_x$ can be interpreted as a function $f$ on the vertices of $\cG$ such that the space of functions generated by $\{T_x^i(f)\}_{i\geq0}$ is finite-dimensional where the Hecke operator $T_x$ is defined by the formula
$$ T_x(f)([g]) \quad = \quad \sum_{\substack{\text{edges }e\text{ with origin }[g]\\\text{and terminus }[g']}} [\Stab_\Gamma([g]):\Stab_\Gamma(e)] \cdot f([g']) $$
for each coset $[g]\in\rquot{\PGL_2(F_x)}{\PGL_2(\cO_x)}$. 

The inclusion of $\PGL_2(F_x)$ as $x$-component into $\PGL_2(\AA)$ induces a map
$$ \lrquot{\Gamma}{\PGL_2(F_x)}{\PGL_2(\cO_x)} \quad \longrightarrow \quad \lrquot{\PGL_2(F)}{\PGL_2(\AA)}{\PGL_2(\cO_\AA)} $$
where $\cO_\AA$ is the maximal compact subring of the adeles $\AA$ of $F$. In the case that $F$ is a rational function field (as in \cite{Gekeler1}, \cite{Gekeler2} and \cite{Gekeler-Nonnengardt}), or, more generally, a function field with odd class number, and $x$ is a place of odd degree, this map is a bijection as a consequence of the strong approximation property of $\SL_2$. The double coset space on the right hand side is the domain of automorphic forms over $F$, and the bijection is equivariant with respect to the Hecke operator $T_x$ and its global equivalent $\Phi_x$. 

In this sense, it is possible to approximate automorphic forms in this case and use the theory from Serre's book. However, the method of approximation breaks down if the function field has even class number or if the Hecke operator of interest is attached to a place of even degree. For automorphic forms over any function field (with possibly even class number) or for the investigation of Hecke operators at any place of a given function field respective a simultaneous description of all Hecke operators, the method of strong approximation is thus insufficient, and we see the need of a global analogon, which is the starting point of this paper.

The applications of this theory are primarily in explicit computations with automorphic forms. For instance in \cite{Lorscheid3}, graphs of Hecke operators are used to calculate the dimensions of spaces of cusp forms and toroidal automorphic forms. From a more conceptual viewpoint, it might be fruitful to explore the connections between graphs of Hecke operators and Drinfel$'$d modules; in particular, it might contribute to the Langland's program since there is a generalisation of graphs of Hecke operator to all reductive groups via adelic Bruhat-Tits buildings, which we forgo to explain here.

We give an overview of the content of this paper. In section \ref{section_definitions}, we introduce the graph of a Hecke operator as a graph with weighted edges that encodes the action of a Hecke operator on automorphic forms. This definition applies to every Hecke operator of $\PGL_2(\AA)$ over a global field. We collect first properties of these graphs and describe, how the algebraic structure of the Hecke algebra is reflected in dependencies between the graphs. In section \ref{section_unramified_hecke_operators}, we describe the graph $\cG_x$ of the unramified Hecke operators $\Phi_x$ (which correspond to the local Hecke operators $T_x$ as introduced above) in terms of coset representatives. In section \ref{section_bruhat-tits}, we make the connection to Bruhat-Tits trees precise: each component of $\cG_x$ is a quotient of $\cT_x$ by a certain subgroup of $\PGL_2(F_x)$, and the components of $\cG_x$ are counted by the $2$-torsion of the class group of $\cO^x_F$. In section \ref{section_vertex_labelling}, we associate to each vertex of $\cG_x$ a coset in $\Cl F/2\Cl F$ where $\Cl F$ is the divisor class group of $F$. We describe how these labels are distributed in $\cG_x$ in dependence of $x$.

In section \ref{section_geometric_interpretation}, we give the vertices and edges of $\cG_x$ a geometric meaning following ideas connected to the geometric Langland's program. Namely, the vertices correspond to the isomorphism classes of $\PP^1$-bundles on the smooth projective curve $X$ with function field $F$, and the edges correspond to certain exact sequences of sheaves on $X$. In section \ref{section_vertices}, we distinguish three classes of rank $2$ bundles: those that decompose into a sum of two line bundles, those that are the trace of a line bundle over the quadratic constant extension $X'$ of $X$ and those that are geometrically indecomposable. This divides the vertices of $\cG_x$ into three subclasses $\PBundec X$, $\PBuntr X$ and $\PBungi X$. The former two sets of vertices have a simple description in terms of the divisor class groups of $X$ and $X'$.

In section \ref{section_reduction_theory}, we introduce the integer valued invariant $\delta$ on the set of vertices, which is closely connected to reduction theory of rank $2$ bundles. This helps us to refine our view on the vertices: $\PBuntr X$ and $\PBungi X$ are contained in the finite set of vertices $v$ with $\delta(v)\leq 2g_X-2$ where $g_X$ is the genus of $X$. 
In section \ref{section_nucleus_and_cusps}. we describe the edges between vertices: $\cG_x$ decomposes into a finite graph, which depends heavily on the arithmetic of $F$, and class number many cusps, which are infinite weighted subgraphs of the form
\begin{center} \includegraphics{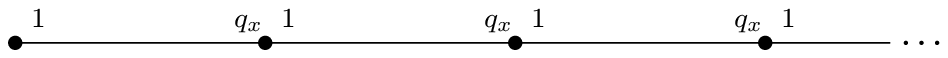}  \end{center}
We conclude this section with a summary of results on $\cG_x$ and illustrate them in Figure \ref{figure_genstr}.

In section \ref{section_automorphic_forms}, we explain, how abstract properties of unramified automorphic forms---with name, the compact support of cusp forms and eigenvalue equations for Eisenstein series---lead to an explicit description of them as functions on the vertices of the graphs $\cG_x$. In section \ref{section_finite-dimensionality}, we show that the spaces of functions on $\Vertex\cG_x$ that satisfy the cuspidal respective toroidal condition are finite dimensional. In particular, these spaces of functions contain only automorphic forms.

In Appendix \ref{appendix_examples}, we give a series of examples for a rational function field: $\cG_x$ for $\deg x\leq5$, the graphs of $\Phi_x^2$ and $\Phi_x^3$ for $\deg x=1$ and the graphs of two ramified Hecke operators. We give short explanations on how to calculate these examples.

\medskip\noindent
\textbf{Acknowledgements:} This paper is extracted from my thesis \cite{Lorscheid-thesis}. First of all, I would like to thank Gunther Cornelissen for his advice during my graduate studies. I would like to thank Frits Beukers and Roelof Bruggeman for their numerous comments on a lecture series about my studies.


\section{Definitions}
\label{section_definitions}

\noindent
In this section, we set up the notations that are used throughout the paper and introduce the notion of a graph of a Hecke operator. We collect first properties of these graphs and describe how the algebraic structure of the Hecke algebra is reflected in dependencies between the graphs of different Hecke operators.

\begin{pg}
 Let $q$ be a prime power and $F$ be the function field of a smooth projective curve $X$ over $\FF_q$. Let $\norm X$ the set of closed points of $X$, which we identify with the set of places of $F$. We denote by $F_x$ the completion of $F$ at $x\in \norm X$ and by $\cO_x$ the integers of $F_x$. We choose a uniformizer $\pi_x\in F$ for every place $x$. Let $\kappa_x=\cO_x/(\pi_x)$ be the residue field. Let $\deg x$ be the degree of $x$ and let $q_x=q^{\deg x}$ be the cardinality of $\kappa_x$. We denote by $\norm \ _x$ the absolute value on $F_x$ resp.\ $F$ such that $\norm{\pi_x}_x=q_x^{-1}$.

 Let $\AA$ be the adele ring of $F$ and $\AA^\times$ the idele group. Put $\cO_\AA=\prod\cO_x$ where the product is taken over all places $x$ of $F$. 
 The idele norm is the quasi-character $\norm\ :\AA^\times\to\CC^\times$ that sends an idele $(a_x)\in\AA^\times$ to the product $\prod\norm{a_x}_x$ over all local norms. By the product formula, this defines a quasi-character on the idele class group $\rquot{\AA^\times}{F^\times}$.

 Let $G=\PGL_2$. Following the habit of literature about automorphic forms, we will often write $G_\AA$ instead of $G(\AA)$ for the group of adelic points and $G_F$ instead of $G(F)$ for the group of $F$-valued points, et cetera. Note that $G_\AA$ comes together with an adelic topology that turns $G_\AA$ into a locally compact group. Let $K=G_{\cO_\AA}$ be the standard maximal compact open subgroup of $G_\AA$. We fix the Haar-measure on $G_\AA$ for which $\vol K=1$.

 The \emph{Hecke algebra $\cH$ for $G_\AA$} is the complex vector space of all compactly supported locally constant functions $\Phi: G_\AA\to\CC$ together with the convolution product
 $$ \Phi_1 \ast \Phi_2: g \mapsto \int\limits_{G_\AA} \Phi_1(gh^{-1})\Phi_2(h) \,dh\,. $$
 A Hecke operator $\Phi\in\cH$ acts on the space $\cV=C^0(G_\AA)$ of continuous functions $f:G_\AA\to\CC$ by the formula
 $$ \Phi(f)(g) \ = \ \int\limits_{G_\AA}\Phi(h)f(gh)\,dh. $$
 Let $K'$ be a compact open subgroup of $G_\AA$. Then we denote by $\cH_{K'}$ the subalgebra of $\cH$ that consists of all bi-$K'$-invariant functions. The above action restricts to an action of $\cH_{K'}$ on $\cV^{K'}$, the space of right $K'$-invariant functions. 
\end{pg}

\begin{lemma}
 \label{Hecke_action_as_finite_sum}
 For every $K'$ and every $\Phi\in\cH_{K'}$, there are $h_1,\dotsc,h_r\in G_\AA$ and $m_1,\dotsc,m_r\in\CC$ for some integer $r$ such that for all $g\in G_\AA$ and all $f\in\cV^{K'}$,
 $$ \Phi(f)(g) \ = \ \sum_{i=1}^r \, m_i\cdot f(gh_i) \;. $$
\end{lemma}

\begin{proof}
 Since $\Phi$ is $K'$-bi-invariant and compactly supported, it is a finite linear combination of characteristic functions on double cosets of the form $K'hK'$ with $h\in G_\AA$. So we may reduce the proof to $\Phi=\ch_{K'hK'}$. Again, since $K'hK'$ is compact, it equals the union of a finite number of pairwise distinct cosets $h_1K',\ldots,h_{r}K'$, and thus
 $$ \int\limits_{G_\AA}\ch_{K'hK'}(h')f(gh')\,dh' \ = \ \sum_{i=1}^{r} \ \ \int\limits_{G_\AA}\ch_{h_iK'}(h')f(gh')\,dh \ = \ \sum_{i=1}^{r} \ \vol(K') f(gh_i) $$
 for arbitrary $g\in G_\AA$.
\end{proof}

 We will write $[g]\in\lrquot{G_F}{G_\AA}{K'}$ for the class that is represented by $g\in G_\AA$. Other cosets will also occur in this paper, but it will be clear from the context what kind of class the square brackets relate to.

\begin{prop}
 For all $\Phi \in \cH_{K'}$ and $[g]\in \lrquot{G_F}{G_\AA}{K'}$, there is a unique set of pairwise distinct classes $[g_1],\ldots,[g_r] \in \lrquot{G_F}{G_\AA}{K'}$ and numbers $m_1,\ldots,m_r \in \CC^\times$ such that for all $f\in\cV^{K'}$,
 $$ \Phi(f)(g) \ = \ \sum_{i=1}^r m_i f(g_i) \;. $$
\end{prop}

\begin{proof}
 Uniqueness is clear, and existence follows from Lemma \ref{Hecke_action_as_finite_sum} after we have taken care of putting together values of $f$ in same classes of $\lrquot{G_F}{G_\AA}{K'}$ and excluding the zero terms.
\end{proof}

\begin{df}
 \label{def_graph}
 With the notation of the preceding proposition we define
 $$ \cU_{\Phi,K'}([g]) = \{([g],[g_i],m_i)\}_{i=1,\ldots,r} \;. $$
 The classes $[g_i]$ are called the {\it $\Phi$-neighbours of $[g]$ (relative to $K'$)}, and the $m_i$ are called their \emph{weights}.
 
 The {\it graph $\cG_{\Phi,K'}$ of $\Phi$ (relative to $K'$)} consists of vertices
 $$ \Vertex \cG_{\Phi,K'} \ = \ \lrquot{G_F}{G_\AA}{K'} $$
 and oriented weighted edges
 $$ \Edge \cG_{\Phi,K'} \ = \ \bigcup_{v \in \Vertex \cG_{\Phi,K'}} \cU_{\Phi,K'}(v) \;. $$
\end{df}

\begin{rem}
 The usual notation for an edge in a graph with weighted edges consists of pairs that code the origin and the terminus, and an additional function on the set of edges that gives the weight. For our purposes, it is more convenient to replace the set of edges by the graph of the weight function and to call the resulting triples that consist of origin, terminus and the weight the edges of $\cG_{\Phi,K'}$.
\end{rem}

\begin{pg}
 We make the following drawing conventions to illustrate the graph of a Hecke operator: vertices are represented by labelled dots, and an edge $(v,v',m)$ together with its origin $v$ and its terminus $v'$ is drawn as 
 \begin{center} \includegraphics{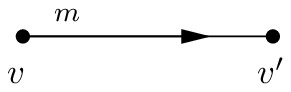} \end{center}
 If there is precisely one edge from $v$ to $v'$ and precisely one from $v'$ to $v$, which we call the inverse edge, we draw \vspace{0,3cm}
 $$ \begin{array}{cccccccc} 
      &\hspace{-0,3cm}\textup{in place of}
	  &&\hspace{-0,3cm}\textup{and}
	  &&\hspace{-0,3cm}\textup{in place of}
	  &&\hspace{-0,4cm}\textup{.} \vspace{-0,7cm} \\
       \xy \xyimport(1,1)(0,-5){\includegraphics{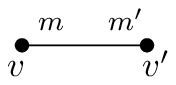}} \endxy 
	  &&\xy \xyimport(1,1)(0,-5){\includegraphics{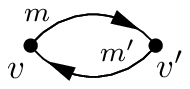}} \endxy 
	  &&\xy \xyimport(1,1)(0,-0){\includegraphics{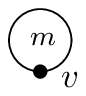}} \endxy 
	  &&\xy \xyimport(1,1)(0,-5){\includegraphics{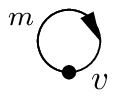}} \endxy 
    \end{array} $$
 There are various examples for rational function fields in Appendix \ref{appendix_examples}, and in \cite{Lorscheid3}, one finds graphs of Hecke operators for elliptic function fields.
\end{pg}

\begin{pg}
 \label{read_off_Hecke}
 We collect some properties that follow immediately from the definition of a graph of a Hecke operator $\Phi$. For $f\in\cV^{K'}$ and $[g]\in \lrquot{G_F}{G_\AA}{K'}$, we have that
 $$ \Phi(f)(g) = \sum_{\begin{subarray}{c} ([g],[g'],m')\\ \in \Edge \cG_{\Phi,K'}\end{subarray}} m' f(g') \;. $$
 Hence one can read off the action of a Hecke operator on $f\in\cV^{K'}$ from the illustration of the graph:
 \begin{center} \includegraphics{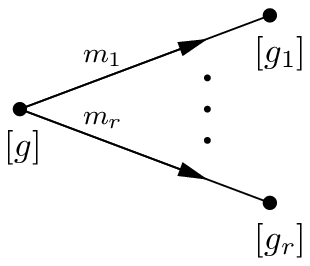} \end{center}

 Since $\cH=\bigcup \cH_{K'}$, with $K'$ running over all compact opens in $G_\AA$, the notion of the graph of a Hecke operator applies to any $\Phi\in\cH$. The set of vertices of the graph of a Hecke operator $\Phi\in\cH_{K'}$ only depends on $K'$, and only the edges depend on the particular chosen $\Phi$. There is at most one edge for each two vertices and each direction, and the weight of an edge is always non-zero. Each vertex is connected with only finitely many other vertices.
 
 \label{algebra_strucure_in_graphs}
 The algebra structure of $\cH_{K'}$ has the following implications on the structure of the set of edges (with the convention that the empty sum is defined as $0$).  For the zero element $0\in\cH_{K'}$, the multiplicative unit $1\in\cH_{K'}$, and arbitrary $\Phi_1,\Phi_2\in\cH_{K'}$, $r\in\CC^\times$ we obtain that
 \begin{eqnarray*}
    \Edge\cG_{0,K'}&=&\emptyset\;, \\  
    \Edge\cG_{1,K'}&=&\bigl\{\,(v,v,1)\,\bigr\}_{v\in \Vertex\cG_{1,K'}}\;, \\
    \Edge\cG_{\Phi_1+\Phi_2,K'}&=& \bigl\{\,(v,v',m)\, \bigl\vert \, m \ = \hspace{-0,4cm}\sum_{(v,v',m')\in\Edge\cG_{\Phi_1,K'}} \hspace{-0,4cm} m' \hspace{0,4cm} + \hspace{-0,4cm}\sum_{(v,v',m'')\in\Edge\cG_{\Phi_2,K'}} \hspace{-0,4cm} m'' \, \neq \, 0\, \bigr\}\;, \\ \\
    \Edge\cG_{r\Phi_1,K'}&=&\bigl\{\, (v,v',rm)\,\bigl\vert\,(v,v',m)\in\Edge\cG_{\Phi_1,K'}\,\bigr\}\;,\text{ and} \\
    \Edge\cG_{\Phi_1\ast\Phi_2,K'}&=&\bigl\{\,(v,v',m)\,\bigl\vert\, m\ = \hspace{-0,8cm}\sum_{\substack{(v,v'',m')\in\Edge\cG_{\Phi_1,K'}\vspace{-1pt}\\ \text{and}\vspace{1pt}\\(v'',v',m'')\in\Edge\cG_{\Phi_2,K'}}} \hspace{-0,8cm} m'\cdot m'' \neq 0\,\bigr\}. \\
 \end{eqnarray*}
 If $K''<K'$ and $\Phi\in\cH_{K'}$, then also $\Phi\in\cH_{K''}$. This implies that we have a canonical map $P:\cG_{\Phi,K''}\to\cG_{\Phi,K'}$, which is given by
 $$\begin{array}{ccc} \Vertex\cG_{\Phi,K''}=\lrquot{G_F}{G_\AA}{K''} & \stackrel{P}{\longrightarrow}  & \lrquot{G_F}{G_\AA}{K'}=\Vertex \cG_{\Phi,K'} \\ \end{array} $$
 and							 
 $$\begin{array}{ccc} \Edge\cG_{\Phi,K''} & \stackrel{P}{\longrightarrow} 
                                          & \Edge \cG_{\Phi,K'} \;. \\ 
                      (v,v',m')        & \longmapsto                   & (P(v),P(v'),m') \end{array} $$
\end{pg}
 
\begin{pg}
 \label{associated_matrix}
 One can also collect the data of $\cG_{\Phi,K'}$ in an infinite-dimensional matrix 
 $M_{\Phi,K'}$,
 which we call {\it the matrix associated to $\cG_{\Phi,K'}$},
 by putting $(M_{\Phi,K'})_{v',v}=m$ if $(v,v',m)\in\Edge\cG_{\Phi,K'}$, and $(M_{\Phi,K'})_{v',v}=0$ otherwise. 
 Thus each row and each column has only finitely many non-vanishing entries.
 
 The properties of the last paragraph imply:
 \begin{eqnarray*}
    M_{0,K'}&=&0, \quad \text{ the zero matrix,} \\
    M_{1,K'}&=&1, \quad \text{ the identity matrix,} \\
	 M_{\Phi_1+\Phi_2,K'}&=&M_{\Phi_1,K'}+M_{\Phi_2,K'}\;,   \\
	 M_{r\Phi_1,K'}&=&rM_{\Phi_1,K'}\;,\quad \text{ and}                      \\
	 M_{\Phi_1\ast\Phi_2,K'}&=&M_{\Phi_2,K'} M_{\Phi_1,K'}\;.
 \end{eqnarray*}
 Let $\cJ(K')\subset \cH_{K'}$ be the ideal of operators that act trivially on $\cV$, then we may regard $\cH_{K'}/\cJ(K')$ as a subalgebra of the algebra of $\CC$-linear maps
 $$ \bigoplus_{\lrquot{G_F}{G_\AA}{K'}} \CC \ \ \ \longrightarrow \ \ \ \bigoplus_{\lrquot{G_F}{G_\AA}{K'}} \CC \;. $$
\end{pg}


\section{Unramified Hecke operators}
\label{section_unramified_hecke_operators}

\noindent
 From now on we will restrict ourselves to unramified Hecke operators, which means, elements in $\cH_K$. In particular, we will investigate the graphs $\cG_x$ of certain generators $\Phi_x$ of $\cH_K$ in more detail.

\begin{pg}
  Consider the uniformizers $\pi_x\in F$  as ideles via the embedding $F^\times\subset F_x^\times\subset\AA^\times$ and define for every place $x$ the unramified Hecke operator $\Phi_x$ as the characteristic function of $K\tinymat \pi_x {} {} 1 K$. It is well-known that $\cH_K\simeq\CC[\Phi_x]_{x\in\norm X}$ as an algebra, which means, in particular, that $\cH_K$ is commutative. By the relations from paragraph \ref{algebra_strucure_in_graphs}, it is enough to know the graphs of generators to determine all graphs of unramified Hecke operators. We use the shorthand notation $\cG_x$ for the graph $\cG_{\Phi_x,K}$, and $\cU_x(v)$ for the $\Phi_x$-neighbours $\cU_{\Phi_x,K}(v)$ of $v$. 

 We introduce the {\it ``lower $x$ convention''} that says that a lower index $x$ on an algebraic group defined over the adeles of $F$ will consist of only the component at $x$ of the adelic points, for example, $G_x=G_{F_x}$. Analogously, we put $K_x=G_{\cO_x}$.

 The {\it ``upper $x$ convention''} means that an upper index $x$ on an algebraic group defined over the adeles of $F$ will consist of all components except for the one at $x$. In particular, we first define $\AA^x=\prod'_{y\neq x} F_y$, the restricted product relative to $\cO^x=\prod_{y\neq x}\cO_y$ over all places $y$ that do not equal $x$. Another example is $G^x=G_{\AA^x}$. We put $K^x=G_{\cO^x}$.
\end{pg}

\begin{pg}\label{def_xi_w}
 We embed $\kappa_x$ via $\kappa_x\subset F_x\subset \AA$, thus an element $b\in\kappa_x$ will be considered as the adele whose component at $x$ is $b$ and whose other components are $0$. Let $\PP^1$ be the projective line. Define for $w\in\PP^1(\kappa_x)$,
 $$ \xi_w \ = \ \smallmat \pi_x b {} 1 \quad \textrm{if }w=[1:b]\hspace{0,5cm} \text{and} \hspace{0,5cm} \xi_w \ = \ \smallmat 1 {} {} \pi_x \quad \textrm{if }w=[0:1]. $$
 It is well-known (cf.\ \cite[Lemma 3.7]{Gelbart1}) that the domain of $\Phi_x$ can be describe as
 $$ K\smallmat \pi_x {} {} 1 K = \coprod_{w\in\PP^1(\kappa_x)} \xi_wK \;. $$
 Consequently the weights of edges in $\cG_x$ are positive integers (recall that $\vol K=1$). We shall also refer to the weights as the \emph{multiplicity} of a $\Phi_x$-neighbour. The above implies the following.
\end{pg}

\begin{prop}
 \label{Phi_x_neighbours}
 The $\Phi_x$-neighbours of $[g]$ are the classes $[g\xi_w]$ with $\xi_w$ as in the previous lemma, and the multiplicity of an edge from $[g]$ to $[g']$ equals the number of $w\in\PP^1(\kappa_x)$ such that $[g\xi_w]=[g']$. The multiplicities of the edges originating in $[g]$ sum up to $\#\ \PP^1(\kappa_x)=q_x+1$.\qed
\end{prop}


\section{Connection with Bruhat-Tits trees}
\label{section_bruhat-tits}

\noindent
Fix a place $x$. In this section we construct maps from Bruhat-Tits trees to $\cG_x$. This will enable us to determine the components of $\cG_x$.

\begin{df}
 \label{def_Bruhat-Tits}
 The {\it Bruhat-Tits tree} $\cT_x$ for $F_x$ is the (unweighted) graph with vertices
 $$ \Vertex \cT_x \ = \ \rquot{G_x}{K_x} $$
 and edges
 $$ \Edge \cT_x \ = \ \{\ ([g],[g'])\mid\exists w\in\PP^1(\kappa_x), \ g \equiv g'\xi_w \pmod{K_x}\ \} \;. $$
\end{df}

\begin{pg}
 \label{Psi_x}
 Consider $G_x$ to be embedded in $G_\AA$ as the component at $x$. For each $h\in G_\AA$, we define a map
 $$ \Psi_{x,h}: \cT_x \longrightarrow \cG_x $$
 by
 $$ \begin{array}{rcl}\Vertex \cT_x = \rquot{G_x}{K_x} & \longrightarrow & \lrquot{G_F}{G_\AA}{K} = \Vertex \cG_x\\ \ [g] \ \ \ \           & \longmapsto     & \ \ \  [hg] \end{array} $$
 and
 $$ \begin{array}{ccc}\Edge \cT_x & \longrightarrow & \Edge \cG_x\\ ([g],[g'])      & \longmapsto     & ([hg],[hg'],m) \end{array} $$
 with $m$ being the number of vertices $[g'']$ that are adjacent to $[g]$ in $\cT_x$ such that $\Psi_{x,h}([g''])=\Psi_{x,h}([g'])$.
 
 By Proposition \ref{Phi_x_neighbours} and the definition of a Bruhat-Tits tree, $\Psi_{x,h}$ is well-defined and {\it locally surjective}, i.e.\ it is locally surjective as a map between the associated simplicial complexes of $\cT_x$ and $\cG_x$ with suppressed weights.

 Since Bruhat-Tits trees are indeed trees (\cite[II.1, Thm.\ 1]{Serre}), hence in particular connected, the image of each $\Psi_{x,h}$ is precisely one component of $\cG_x$, i.e.\ a subgraph that corresponds to a connected component of the associated simplicial complex.

 \label{sym_edges}
 Every edge of the Bruhat-Tits tree has an inverse edge, which implies the analogous statement for the graphs $\cG_x$. Namely, if $(v,v',m)\in\Edge\cG_x$, then there is a $m'\in\CC^\times$ such that $(v',v,m')\in\Edge\cG_x$.
\end{pg}

\begin{rem}
 This symmetry of edges is a property that is particular to unramified Hecke operators for $G=PGL_2$. In case of ramification, the symmetry is broken, cf.\ Example \ref{ramification_example}. 
\end{rem}

\begin{pg}
 The algebraic group $\SL_2$ has the \textit{strong approximation property}, i.e.\ for every place $x$, $\SL_2 F$ is a dense subset of $\SL_2\AA^x$ with respect to the adelic topology. This was proven by Kneser (\cite{Kneser2}) for number fields and was extended independently by Prasad (\cite{Prasad}) and Margulis (\cite{Margulis}) to global fields. See \cite[Thm.\ E.2.1]{Laumon2} for a direct reference. We explain, which implication this has on $\PGL_2$. More detail for the outline in this paragraph can be found in \cite[(2.1.3)]{Put-Reversat}.
 
 \label{decomposition_gamma_skg_x}
 Let $x$ be a place of degree $d$. In accordance to the the upper $x$ convention, let $\cO^x=\prod_{y\neq x}\cO_y$. The determinant map on $\GL_2$ induces a bijection on double cosets:
 $$ \lrquot{\GL_2(F)}{\GL_2(\AA^x)}{\GL_2(\cO^x)} \stackrel{\det}{\longrightarrow} \lrquot{F^\times}{(\AA^x)^\times}{(\cO^x)^\times}. $$
 The quotient group $\lrquot{F^\times}{(\AA^x)^\times}{(\cO^x)^\times}$ is nothing else but the ideal class group $\Cl \cO_F^x$ of the {\it integers $\cO_F^x = \bigcap_{y\neq x}(\cO_y\cap F)$ coprime to $x$}. Let $\Cl F=\lrquot{F^\times}{\AA^\times}{\cO_\AA^\times}$ be the divisor class group of $F$ and $\Cl^0 F=\{[a]\in\Cl F|\deg a=0\}$ be the ideal class group. Then we have bijections
 $$ \lrquot{\GL_2(F)}{\GL_2(\AA^x)}{\GL_2(\cO^x)} \ \simeq \ \lrquot{F^\times}{(\AA^x)^\times}{(\cO^x)^\times} \ \simeq \ \Cl \cO_F^x \ \simeq \ \Cl^0 F \times \ZZ/d\ZZ \;. $$
 Let $S\subset \GL_2(\AA^x)$ be a set of representatives for $\lrquot{\GL_2(F)}{\GL_2(\AA^x)}{\GL_2(\cO^x)}$. Then every $g=g^xg_x\in \GL_2(\AA)$ (with $g^x\in \GL_2(\AA^x)$ and $g_x\in \GL_2(F_x)$), there are $s\in S$, $\gamma\in \GL_2(F)$ and $k\in \GL_2(\cO^x)$ such that $g=\gamma s k \tilde g_x$ where $\gamma s k$ equals $g$ in all components $z\neq x$ and $\tilde g_x=\gamma^{-1}g_x$. The condition $[\det s]=[\det g^x]$ as cosets in $\lrquot{F^\times}{(\AA^x)^\times}{(\cO^x)^\times}$ implies that $s\in S$ is uniquely determined by $g^x$. Let $Z$ be the center of $\GL_2$, then
 \begin{multline*} \rquot{\GL_2(\AA)}{\GL_2(\cO_\AA)Z_x} \ = \ \bigl(\rquot{\GL_2(\AA^x)}{\GL_2(\cO^x)}\bigr)\times\bigl(\rquot{G_x}{K_x}\bigr) \\ = \ \bigl(\rquot{\GL_2(\AA^x)}{\GL_2(\cO^x)}\bigr)\times\Vertex\cT_x \;. \end{multline*}
 Define $\Gamma_s = \GL_2(F) \cap s\GL_2(\cO^x)s^{-1}$. Then we obtain the following (cf.\ \cite[(2.1.3)]{Put-Reversat}).
\end{pg}

\begin{prop}
 \label{strong_approx}
  The decomposition $ g=\gamma s k \tilde g_x $ induces a bijective map
 \begin{eqnarray*} \lrquot{\GL_2(F)}{\GL_2(\AA^x)}{\GL_2(\cO_\AA)Z_x}&\longrightarrow&\coprod_{s\in S}\ \lquot{\Gamma_s}{\Vertex\cT_x}\;. \\  \ [g] \                 & \longmapsto   & (s,[\tilde g_x]) \end{eqnarray*} 
 Its inverse is obtained by putting together the components $s\in \GL_2(\AA^x)$ and $\tilde g_x\in G_x$.\qed
\end{prop}

\begin{rem}
 \label{comp_with_Serre}
 On the right hand side of the bijection in Proposition \ref{strong_approx}, we have a finite union of quotients of the form $\lquot{\Gamma_s}{\Vertex \cT_x}$. If $s$ is the identity element $e$, then $\Gamma=\Gamma_e=\GL_2(\cO_F^x)$ is an arithmetic group of the form that Serre considers in \cite[II.2.3]{Serre}. For general $s$, however, I am not aware of any results about $\lquot{\Gamma_s}{\Vertex \cT_x}$ in the literature.
\end{rem}

\begin{pg}
 \label{action_of_centrum}
 So far, we have only divided out the action of the $x$-component $Z_x$ of the centre. We still have to consider the action of $Z^x$. If we restrict the determinant map to the centre and write $J=\{z\in \lrquot{Z_F}{Z^x}{Z_{\cO^x}}\mid \norm{\det z} = 1 \}$, then we have an exact sequence of abelian groups
 $$ \begin{array}{ccccccccccc} 1 & \to & J & \to & \lrquot{Z_F}{Z^x}{Z_{\cO^x}} & \stackrel{\det}{\longrightarrow} &\Cl \cO_F^x & \to & \rquot{\Cl \cO_F^x}{2\Cl \cO_F^x} & \to & 0 \;. \end{array} $$
 Let $S$ be as in paragraph \ref{decomposition_gamma_skg_x}. The action of $Z^x$ on $S$ factors through $2\Cl \cO_F^x$ and the action of $Z^x$ on $\lquot{\Gamma_s}{\Vertex\cT_x}$ factors through $J$ for each $s\in S$. If we let $S'\subset G^x$ be a set of representatives for $\rquot{\Cl \cO_F^x}{2\Cl \cO_F^x}$ (with respect to the determinant map), and $h_2=\#(\Cl F)[2]$ the cardinality of the $2$-torsion, then we obtain:
\end{pg}

\begin{prop}
 \label{strong_approx+comp}
 The decomposition $ g=\gamma s k \tilde g_x $ induces a bijective map
 $$ \lrquot{G_F}{G_\AA}{K} \ \ \longrightarrow \ \ \coprod_{s\in S'}\ \lquot{J \,\Gamma_s}{\Vertex \cT_x}\;. $$
 The inverse maps an element $(s,[\tilde g_x])$ to the class of the adelic matrix with components $s\in G^x$ and $\tilde g_x\in G_x$. The number of components of $\cG_x$ equals
 $$ \#\bigl(\rquot{\Cl\cO_F^x}{2\Cl\cO_F^x}\bigr) = \#(\Cl O_F^x)[2] = \left\{ \begin{array}{ll} h_2  & \textrm{if }\deg x\textrm{ is odd,}\\ 2h_2 & \textrm{if }\deg x\textrm{ is even.} \end{array}\right. $$
\end{prop}

\begin{proof}
 Everything follows from Proposition \ref{strong_approx} and paragraph \ref{action_of_centrum} except for the two equalities in the last line. Regarding the former, observe that both dividing out the squares and taking $2$-torsion commutes with products, so by the structure theorem of finite abelian groups, we can reduce the proof to groups of the form $\ZZ/\tilde p^m \ZZ$ with $\tilde p$ prime. If $\tilde p \neq 2$, then every element is a square and there is no $2$-torsion, hence the equality holds. If $\tilde p = 2$, then $\ZZ/\tilde p^m \ZZ$ modulo squares has one nontrivial class, and there is exactly one nontrivial element in $\ZZ/\tilde p^m \ZZ$ that is $2$-torsion.
 
 Regarding the latter equality, we have that $\Cl \cO_F^x \simeq \Cl^0 F \times \ZZ/d\ZZ$, where $d=\deg x$. As above, $\ZZ/d\ZZ$ modulo squares has a nontrivial class if and only if $d$ is even, and in this case there is only one such class.
\end{proof}


\section{A vertex labelling}
\label{section_vertex_labelling}

\noindent
In this section, we associate to each vertex of $\cG_x$ an element of $\rquot{\Cl F}{2\Cl F}$ and determine how these labels are distributed over the components of $\cG_x$.

\begin{pg}
 Let $\cQ_\AA=\langle a^2\,|\,a\in\AA^\times\rangle$ be the subgroup of squares. We look once more at the determinant map
 $$ \Vertex \cG_x \ = \ \lrquot{G_F}{G_\AA}{K} \ \stackrel{\det}{\longrightarrow} \ \lrquot{F^\times}{\AA^\times}{\cO_\AA^\times \cQ_\AA} \ \simeq \ \rquot{\Cl F}{2\Cl F} \;. $$
 This map assigns to every vertex in $\cG_x$ a label in $\rquot{\Cl F}{2\Cl F}$. Note that $\Cl F/2\Cl F$ has $2h_2$ elements where $h_2=\#(\Cl F)[2]$ for the same reason as used in the proof of Proposition \ref{strong_approx+comp}.
\end{pg}

\begin{prop}
 If the prime divisor $x$ is a square in the divisor class group then all vertices in the same component of $\cG_x$ have the same label, and there are $2h_2$ components, each of which has a different label. Otherwise, the vertices of each component have one of two labels that differ by $x$ in $\rquot{\Cl F}{2\Cl F}$, and two adjacent vertices have different labels, so each connected component is bipartite.
\end{prop}

\begin{proof}
 First of all, observe that each label is realised, since if we represent a label by some idele $a$, then the vertex represented by $\tinymat a {} {} 1 $ has this label.

 Let $\cQ_x=\langle b^2\,\mid\,b\in F_x^\times\rangle$ and $\Cl F_x = \rquot{F_x^\times}{\cO_x^\times}$, a group isomorphic to $\ZZ$. For the Bruhat-Tits tree $\cT_x$, the determinant map 
 $$ \Vertex\cT_x \ = \ \rquot{G_x}{K_x} \ \stackrel{\det}{\longrightarrow} \ \rquot{F_x^\times}{\cO_x^\times \cQ_x} \simeq \ \rquot{\Cl F_x}{2\Cl F_x} \ \simeq \ \ZZ/2\ZZ  $$
 defines a labelling of the vertices, and the two classes of $\rquot{F_x^\times}{\cO_x^\times \cQ_x}$ are represented by $1$ and $\pi_x$. Two adjacent vertices have the different labels since for $g\in G_x$ and $\xi_w$ as in Definition \ref{def_Bruhat-Tits}, $\det(g\xi_w)=\pi_x\det g$ represents a class different from $\det g$ in $\Vertex \cT_x$.

 Define for $a\in\AA^\times$ a map $\psi_{x,a}: \rquot{F_x^\times}{\cO_x^\times \cQ_x} \ \to \ \lrquot{F^\times}{\AA^\times}{\cO_\AA^\times \cQ_\AA}$ by $\psi_{x,a}([b])=[ab]$, where $b$ is viewed as the idele concentrated in $x$. For every $h\in G_\AA$ we obtain a commutative diagram
 $$ \xymatrix{\Vertex \cT_x  \ar[d] \ar@{}|{\displaystyle =}[r] & \rquot{G_x}{K_x} \ar[d]^{\det} \ar[r]^{\Psi_{x,h}\hspace{0.35cm} }& \lrquot{G_F}{G_\AA}{K} \ar[d]^{\det} \ar@{}|{\displaystyle =}[r] & \Vertex \cG_x \ar[d] \\
              \rquot{\Cl F_x}{2\Cl F_x} \ar@{}|{\displaystyle \simeq}[r] & \rquot{F_x^\times}{\cO_x^\times \cQ_x} \ar[r]^{\psi_{x,\det h}\hspace{0.35cm} } & \lrquot{F^\times}{\AA^\times}{\cO_\AA^\times \cQ_\AA} \ar@{}|{\displaystyle \simeq}[r]  & \rquot{\Cl F}{2\Cl F}                                                                       \;. } $$
 This means that vertices with equal labels map to vertices with equal labels.
 
 Each component of $\cG_x$ lies in the image of a suitable $\Psi_{x,h}$, thus has at most two labels.
 On the other hand, the two labels of $\cT_x$ map to $\psi_{x,\det h}([1])=[a]$ and $\psi_{x,\det h}([\pi_x])=[a\pi_x]$. 
 The divisor classes of $[a]$ and $[a\pi_x]$ differ by the class of the prime divisor $x$, and are equal if and only if
 $x$ is a square in the divisor class group. If so, according to Proposition \ref{strong_approx+comp},
 there must be $2h_2$ components so that the $2h_2$ labels are spread over all components. 
 If $x$ is not a square then by the local surjectivity
 of $\Psi_{x,h}$ on edges two adjacent vertices of $\cG_x$ also have different labels.
\end{proof}





\section{Geometric interpretation of unramified Hecke operators}
\label{section_geometric_interpretation}

\noindent
 A fundamental observation in the geometric Langland's program (for $\PGL_2$, in this case) is that the domain of automorphic forms (with a certain ramification level) corresponds to the isomorphism classes of $\PP^1$-bundles (with a corresponding level structure). The action of Hecke operators can be given a geometric meaning, which makes it possible to let algebraic geometry enter the field. We will use this geometric view point for a closer examination of the graphs of unramified Hecke operators. We begin with recalling the geometric interpretation of unramified Hecke operators. For more reference, see \cite{Gaitsgory}.

\begin{pg}
 Let $\cO_X$ be the structure sheaf of the smooth projective curve $X$ and $\eta$ the generic point. We can identify the stalks $\cO_{X,x}$ of the structure sheaf $\cO_X$ at closed points $x\in\norm X$ and their embeddings into the generic stalk $\cO_{X,\eta}$ with
 $$ \cO_{X,x} \ \simeq \ \cO_x \cap F \ \longhookrightarrow \ F \ \simeq \ \cO_{X,\eta} \;. $$

 \label{det_and_deg}
 We identify vector bundles on $X$ with the corresponding locally free sheaf (\cite[Ex. II.5.18]{Hartshorne}). We denote by $\Bun_nX$ the set of isomorphism classes of {\it rank $n$ bundles} over $X$ and by $\Pic X$ the {\it Picard group}. For $\cL_1,\cL_2\in\Pic X$, we use the shorthand notation $\cL_1\cL_2$ for $\cL_1\otimes\cL_2$. The group $\Pic X$ acts on $\Bun_nX$ by tensor products. Let $\PBun_nX$ be the orbit set $\Bun_nX\,/\,\Pic X$, which is nothing else but the set of isomorphism classes of $\PP^{n-1}$-bundles over $X$ (\cite[Ex. II.7.10]{Hartshorne}). 

 We will call the elements of $\PBun_2 X$ {\it projective line bundles}. If we regard the total space of a projective line bundle as a scheme, then we obtain nothing else but a ruled surface, cf.\ \cite[Prop. V.2.2]{Hartshorne}. Thus $\PBun_2 X$ may also be seen as the set of isomorphism classes of ruled surfaces over $X$.

 If two vector bundles $\cM_1$ and $\cM_2$ are in the same orbit of the action of $\Pic X$, we write $\cM_1 \sim\cM_2$ and say that $\cM_1$ and $\cM_2$ are {\it projectively equivalent}. By $[\cM]\in\PBun_2X$, we mean the class that is represented by the rank $2$ bundle $\cM$.
 
 Let $\Cl X=\Cl F$ be the divisor group of $X$. Every divisor $D\in\Cl X$ defines the \emph{associated line bundle} $\cL_D$, which defines an isomorphism $\Cl X\to\Pic X$ of groups (\cite[Prop. II.6.13]{Hartshorne}). The degree $\deg\cM$ of a vector bundle $\cM$ with $\det \cM \simeq \cL_D$ is defined as $\deg D$. For a torsion sheaf $\cF$, the degree is defined by $\deg \cF = \sum_{x\in\norm X} \dim_{\FF_q}(\cF_x)$. The degree is additive in short exact sequences.
\end{pg}

\begin{rem}
 Note that if $D=x$ is a prime divisor, the notation for the associated line bundle $\cL_x$ coincides with the notation for the stalk of $\cL$ at $x$. In order to avoid confusion, we will reserve the notation $\cL_x$ strictly for the associated line bundle. In case we have to consider the stalk of a line bundle, we will use a symbol different from $\cL$ for the line bundle.
\end{rem}

\begin{pg}
 \label{GL_n-Bun_n}\label{bundle_to_GL_n}
 We associate to every $g=(g_x)\in\GL_2(\AA)$ the rank $2$ bundle $\cM_g$ that is defined by the embeddings $g_x^{-1}:\cO_{X,x}^2\to F^2$  of the stalks $(\cM_g)_x=\cO_{X,x}^2$ at closed points $x$ into the generic stalk $(\cM_g)_\eta=F^2$. This association induces a bijection
 $$ \begin{array}{ccc}
     \lrquot{\GL_2(F)}{\GL_2(\AA)}{\GL_2(\cO_\AA)} & \overset{1:1}{\longleftrightarrow} & \Bun_2X \\
	                          {[g]}              & \longmapsto                    & \cM_g
	 \end{array} $$
 such that $\cM_g\otimes\cL_a=\cM_{ag}$ for $a\in\AA^\times$, and $\deg \cM_g = \deg(\det g)$. Consequently, there is a bijection
 $$ \lrquot{G_F}{G_\AA}{K} \quad \overset{1:1}{\longleftrightarrow} \quad \PBun_2X, $$
 which allows us to identify the vertex set $\Vertex\cG_x=\lrquot{G_F}{G_\AA}{K}$ with $\PBun_2X$.
\end{pg}

\begin{pg}
 \label{def_cK_x}
 The next task is to describe edges of $\cG_x$ in geometric terms. We say that two exact sequences of sheaves
 $$ 0\to\cF_1\to\cF\to\cF'_1\to0 \hspace{1cm} \textrm{and} \hspace{1cm} 0\to\cF_2\to\cF\to\cF'_2\to0\;, $$
 are {\it isomorphic with fixed $\cF$} if there are isomorphisms $\cF_1\to\cF_2$ and $\cF'_1\to\cF'_2$ such that
 $$ \xymatrix{0\ar[r]&{\cF_1}\ar[r]\ar[d]^\simeq&{\cF}\ar[r]\ar@{=}[d]&{\cF'_1}\ar[r]\ar[d]^\simeq&0\\ 0\ar[r]&{\cF_2}\ar[r]&{\cF}\ar[r]&{\cF'_2}\ar[r]&0} $$
 commutes.
 
 Let $\cK_x$ be the torsion sheaf that is supported at $x$ and has stalk $\kappa_x$ at $x$, where $\kappa_x$ is the residue field at $x$. Fix a representative $\cM$ of $[\cM]\in\PBun_2 X$. Then we define $m_x([\cM],[\cM'])$ as the number of isomorphism classes of exact sequences
 $$ \ses{\cM''}{\cM}{\cK_x} \;, $$
 with fixed $\cM$ and with $\cM''\sim\cM'$. This number is independent of the choice of the representative $\cM$ because for another choice, which would be a vector bundle of the form $\cM\otimes\cL$ for some $\cL\in\Pic X$, we have the bijection
 $$ \begin{array}{ccc}
    \left\{\begin{array}{c} \text{isomorphism classes}\\0\to\cM''\to\cM\to\cK_x\to0\\ \text{with fixed }\cM\end{array}\right\} \vspace{0.5cm}&\quad\longrightarrow\quad & \left\{\begin{array}{c}\text{isomorphism classes}\\0\to\cM'''\to\cM\otimes\cL\to\cK_x\to0\\ \text{with fixed }\cM\otimes\cL\end{array}\right\} \;. \\ (0\to\cM''\to\cM\to\cK_x\to0) & \longmapsto & (0\to\cM''\otimes\cL\to\cM\otimes\cL\to\cK_x\to0) 
    \end{array} $$
\end{pg}

\begin{df}
 Let $x$ be a place. For a projective line bundle $[\cM]\in\PBun_2X$ we define
 $$ \cU_x([\cM]) \ = \ \left\{ ([\cM],[\cM'],m) \ \left|\ m=m_x([\cM],[\cM'])\neq0 \right.\right\} \;, $$
 and call the occurring $[\cM']$ the {\it $\Phi_x$-neighbours of $[\cM]$}, and $m_x([\cM],[\cM'])$ their {\it multiplicity}.
\end{df}

\begin{pg}
 \label{geom_nb}
 We shall show that this concept of neighbours is the same as the one defined for classes in $\lrquot{G_F}{G_\AA}{K}$ (Definition \ref{def_graph}). Recall that in Proposition \ref{Phi_x_neighbours}, we determined the $\Phi_x$-neighbours of a class $[g]\in\lrquot{G_F}{G_\AA}{K}$ to be of the form $[g\xi_w]$ for a $w\in\PP^1(\kappa_x)$. The elements $\xi_w$ define exact sequences
 $$ \xymatrix{ 0 \ar[r] & *!<0pt,5pt>{\prod\limits_{y\in\norm X} \cO_{X,y}^2} \ar[r]^{\xi_w} & *!<0pt,5pt>{\prod\limits_{y\in\norm X} \cO_{X,y}^2} \ar[r] & *!<0pt,1pt>{\ \kappa_x} \ar[r] & 0 } \;, $$
  of $\FF_q$-modules and consequently an exact sequence
 $$ \xymatrix{ 0 \ar[r] & {\cM_{g\xi_w}} \ar[r] &{\cM_g} \ar[r] & {\cK_x} \ar[r] & 0 } \;. $$
 of sheaves, where $\cM_{g\xi_w}$ and $\cM_g$ are the rank $2$ bundles associated to $g\xi_w$ resp.\ $g$.
 This maps $w\in\PP^1(\kappa_x)$ to the isomorphism class of $\bigr( 0\to\cM_{g\xi_w}\to\cM_g\to\cK_x\to0 \bigl)$ with fixed $\cM_g$. On the other hand, as we have chosen a basis for the stalk at $x$, each isomorphism class 
 of sequences $\bigr( 0\to\cM'\to\cM\to\cK_x\to0 \bigl)$ with fixed $\cM$ defines an element in 
 $\PP\bigl(\rquot{\cO_{X,x}^2}{(\pi_x\cO_{X,x}})^2\bigr)=\PP^1(\kappa_x)$, which gives back $w$. 

 \label{geometric_neighbourhood}
 Thus for every $x\in\norm X$, the map
 $$ \begin{array}{ccc}
     \cU_x([g]) & \longrightarrow & \cU_x([\cM_g]) \\
	  ([g],[g'],m)  & \longmapsto     & ([\cM_g],[\cM_{g'}],m)
	 \end{array} $$
 is a well-defined bijection. We finally obtain the geometric description of the graph $\cG_x$ of $\Phi_x$.
\end{pg}

\begin{prop}
 \label{geometric_graph}
 Let $x\in\norm X$. The graph $\cG_x$ of $\Phi_x$ is described in geometric terms as
 \begin{align*} 
   \Vertex \cG_x \ &= \ \PBun_2 X\quad \textrm{ and} \\
   \Edge \cG_x \ &= \ \coprod_{[\cM]\in\PBun_2X} \cU_x([\cM]) \;. \mbox{\qed} 
 \end{align*}
\end{prop}

\begin{rem}
 \label{connection_to_Trees}
 This interpretation shows that the graphs that we consider are a global version of the graphs of Serre (\cite[Chapter II.2]{Serre}). We are looking at all rank $2$ bundles on $X$ modulo the action of the Picard group of $X$ while Serre considers rank $2$ bundles that trivialise outside a given place $x$ modulo line bundles that trivialise outside $x$. As already explained in Remark \ref{comp_with_Serre}, we obtain a projection of the graph of Serre to the component of the trivial class $c_0$.

 Serre describes his graphs as quotients of Bruhat-Tits trees by the action of the group $\Gamma=G_{\cO_F^x}$ on both vertices and edges. This leads in general to multiple edges between vertices in the quotient graph, see e.g.\ \cite[2.4.2c]{Serre}. This does not happen with graphs of Hecke operators: there is at most one edge with given origin and terminus. 

 Relative to the action of $\Gamma$ on Serre's graphs, one can define the weight of an edge as the order of the stabiliser of its origin in the stabiliser of the edge. The projection from Serre's graphs to graphs of Hecke operators identifies all the different edges between two vertices, adding up their weights to obtain the weight of the image edge.
\end{rem}


\section{Description of vertices}
\label{section_vertices}

\noindent
 The aim of this section is to show that the set of isomorphism classes of projective line bundles over $X$ can be separated into subspaces corresponding to certain quotients of the the divisor class group of $F$, the divisor class group of $\FF_{q^2} F$ and geometrically indecomposable projective line bundles. We recall a series of facts about vector bundles.

\begin{pg}
 \label{facts_on_vb}
 A vector bundle $\cM$ is {\it indecomposable} if for every decomposition $\cM=\cM_1\oplus\cM_2$ into two subbundles $\cM_1$ and $\cM_2$, one factor is trivial and the other is isomorphic to $\cM$. The Krull-Schmidt theorem holds for the category of vector bundles over $X$, i.e.\ every vector bundle $\cM$ on $X$ defined over $\FF_q$ has, up to permutation of factors, a unique decomposition into a direct sum of indecomposable subbundles, see \cite[Thm. 2]{Atiyah2}.

 The map $p: X'=X\otimes \FF_{q^i} \to X$ defines the \emph{inverse image} or the \emph{constant extension} of vector bundles
 $$ \begin{array}{cccc}
     p^\ast: & \Bun_n X & \longrightarrow & \Bun_n X' \;. \\
	          & \cM      & \longmapsto     & p^\ast\cM
    \end{array} $$
 The isomorphism classes of rank $n$ bundles that after extension of constants to $\FF_{q^i}$ become isomorphic to $p^\ast\cM$ are classified by $H^1\bigr(\Gal(\FF_{q^i}/\FF_q),\Aut(\cM\otimes\FF_{q^i})\bigl)$, cf.\ \cite[Section 1]{Arason-Elman-Jacob}. The algebraic group $\Aut(\cM\otimes\FF_{q^i})$ is an open subvariety of the connected algebraic group $\End(\cM\otimes\FF_{q^i})$, and thus it is itself a connected algebraic group. As a consequence of Lang's theorem (\cite[Cor.\ to Thm.\ 1]{Lang3}), we have $H^1\bigr(\Gal(\FF_{q^i}/\FF_q),\Aut(\cM\otimes\FF_{q^i})\bigl)=1$.
 
 Thus $p^\ast$ is injective. In particular, one can consider the constant extension to the geometric curve $\overline X=X\otimes\overline\FF_q$ over an algebraic closure $\overline\FF_q$ of $\FF_q$. Then two vector bundles are isomorphic if and only if they are geometrically isomorphic, i.e.\ that their constant extensions to $\overline X$ are isomorphic. We can therefore think of $\Bun_n X$ as a subset of $\Bun_n X'$ and $\Bun_n \overline X$.
 
 On the other hand, $p:X'\to X$ defines the direct image or the {\it trace} of vector bundles
 $$ \begin{array}{cccc}
     p_\ast: & \Bun_n X' & \longrightarrow & \Bun_{ni} X \;. \\
          	 & \cM       & \longmapsto     & p_\ast \cM
	 \end{array} $$ 
 We have for $\cM \in \Bun_n X$ that $ p_\ast p^\ast \cM\simeq\cM^i$ and for $\cM\in\Bun_nX'$ that $p^\ast p_\ast \cM \simeq \bigoplus \cM^\tau$ where $\tau$ ranges over $\Gal(\FF_{q^i}/\FF_q)$ and $\cM^\tau$ is defined by the stalks $\cM^\tau_x=\cM_{\tau^{-1}(x)}$.
 
 We call a vector bundle {\it geometrically indecomposable} if its extension to $\overline X$ is indecomposable. In \cite[Thm. 1.8]{Arason-Elman-Jacob}, it is shown that every indecomposable vector bundle over $X$ is the trace of an geometrically indecomposable bundle over some constant extension $X'$ of $X$.

 There are certain compatibilities of the constant extension and the trace with tensor products. Namely, for a vector bundle $\cM$ and a line bundle $\cL$ over $X$, we have $ p^\ast(\cM\otimes\cL) \simeq p^\ast\cM \otimes p^\ast\cL$
 and for a vector bundle $\cM'$ over $X'$, $ p_\ast\cM' \otimes \cL \simeq p_\ast(\cM'\otimes p^\ast\cL)$. Thus $p^\ast$ induces a map, denoted by the same symbol,
 $$ \begin{array}{cccc}
     p^\ast: & \PBun_n X & \longrightarrow & \PBun_n X' \;, \\
	          & {[\cM]}   & \longmapsto     & [p^\ast\cM]
	 \end{array} $$
 and $p_\ast$ induces
 $$ \begin{array}{cccc}
     p_\ast: & \Bun_n X'\,/\,p^\ast\Pic X & \longrightarrow & \PBun_{ni} X \;. \\
	          & {[\cM]}       & \longmapsto     & [p_\ast \cM]
	 \end{array} $$ 
\end{pg}

\begin{pg}
 \label{labelling_vertices}
 We look at the situation for $n=2$ and $i=2$. Let $\sigma$ be the nontrivial automorphism of $\FF_{q^2} / \FF_q$. The set $\PBun_2 X$ is the disjoint union of the set of classes of decomposable rank $2$ bundles, i.e.\ rank $2$ bundles that are isomorphic to the direct sum of two line bundles, and the set of classes of indecomposable bundles. We denote these sets by $\PBundec X$ and $\PBunindec X$, respectively. Let $\PBungi X\subset\PBunindec X$ be the subset of classes of geometrically indecomposable bundles. Since the rank is $2$, the complement $\PBuntr X=\PBunindec X-\PBungi X$ consists of classes of traces $p_\ast\cL$ of certain line bundles $\cL\in\Pic X'$ that are defined over the quadratic extension $X'=X\otimes\FF_{q^2}$. More precisely, $p_\ast \cL$ decomposes if and only if $\cL\in p^\ast\Pic X$, and then $p_\ast \cL \sim \cO_X \oplus \cO_X$. Thus, we have a disjoint union
 $$ \PBun_2 X \ = \ \PBundec X \ \amalg \ \PBuntr X \ \amalg \ \PBungi X \;. $$
 
 For $[D]\in\Cl X$, define
 $$ c_D \ = \ [\cL_D \oplus \cO_X] \ \in \ \PBundec X \;, $$
 and for a $[D]\in\Cl X'$, define
 $$ t_D \ = \ [p_\ast\cL_D] \ \in \ \PBuntr X \cup \{ c_0 \} \;. $$ 
 Note that $\sigma$ acts on $\Cl X'$ in a way compatible with the identification $\Cl X' \simeq \Pic X'$. 
 Since $p^\ast p_\ast(\cL)\simeq\cL\oplus\cL^\sigma\simeq p^\ast p_\ast(\cL^\sigma)$ for $\cL\in\Pic X'$, 
 and isomorphism classes of vector bundles are stable under constant extensions, we have $t_D = t_{\sigma D}$.

 We derive the following characterisations of $\PBundec X$ and $\PBuntr X$:
\end{pg}

\begin{prop}
 \label{char_PBundec}
 $$ \begin{array}{ccc}
     \Cl X & \longrightarrow & \PBundec X \\
	  {[D]}   & \longmapsto     & c_D
	 \end{array} $$
 is surjective with fibres of the form $\{[D],[-D]\}$.
\end{prop}

\begin{proof}
 Let $\cM$ decompose into $\cL_1\oplus\cL_2$. Then
 $$ \cM \ \simeq \ \cL_1\oplus\cL_2 \ \sim \ \bigl(\cL_1\oplus\cL_2\bigr)\otimes\cL_2^{-1} \ \simeq \ \cL_1\cL_2^{-1} \oplus \cO_X \;, $$
 thus surjectivity follows. Let $\cL_{D'}\oplus\cO_X$ represent the same projective line bundle
 as $\cL_D \oplus \cO_X$, then there is a line bundle $\cL_0$ such that
 $$ \cL_D \oplus \cO_X \ \simeq \ \bigl(\cL_{D'}\oplus\cO_X\bigr)\otimes\cL_0 \;, $$
 and thus either $\cL_0\simeq\cO_X$ and $\cL_D \simeq \cL_{D'}$ or $\cL_0\simeq\cL_D$ and $\cL_{D'}\otimes\cL_D \simeq \cO_X$.
 Hence $[D']$ equals either $[D]$ or $[-D]$.
\end{proof} 

\begin{prop}
 \label{char_PBuntr}
 $$ \begin{array}{ccc}
     \Cl X'\,/\,\Cl X & \longrightarrow & \PBuntr X \cup \{c_0\} \\
	  {[D]}            & \longmapsto     & t_D
	 \end{array} $$
 is surjective with fibres of the form $\{[D],[-D]\}$.  
\end{prop}

\begin{proof}
 From the previous considerations it is clear that this map is well-defined and surjective. Assume that $[D_1], [D_2]\in\Cl X'$ have the same image, then there is a $\cL_0 \in \Pic X$ such that
 $$ p_\ast\cL_1 \ \simeq \ p_\ast\cL_2 \otimes \cL_0 \;, $$
 where we briefly wrote $\cL_i$ for $\cL_{D_i}$. Then in $\PBun_2 X'$, we see that
 \begin{align*} 
  \cL_1 \oplus \cL_1^\sigma \ &\simeq \  p^\ast p_\ast\cL_1 \\ 
  &\simeq \ p^\ast p_\ast\cL_2 \otimes p^\ast \cL_0 \\ 
  &\simeq \ (\cL_2\otimes p^\ast\cL_0)\oplus(\cL_2^\sigma\otimes p^\ast\cL_0) \;,
 \end{align*}  
 thus either $\cL_1 \simeq \cL_2 \otimes p^\ast \cL_0$, which implies that $D_1$ and $D_2$ represent the same class in $\Cl X'\,/\,\Cl X$, or $\cL_1 \simeq \cL_2^\sigma \otimes p^\ast \cL_0$, which means that $D_1$ represents the same class as $\sigma D_2$. But in $\Cl X'\,/\,\Cl X$,
 $$ [\sigma D_2] \ = \ [\, \underbrace{\hspace{-0,2cm}\begin{array}{c}\vspace{-0,4cm}\\ \sigma D_2+D_2\\ \vspace{-0,4cm}\end{array}\hspace{-0,2cm}}_{\in\Cl X} \,-D_2] \ = \ [-D_2] \;.\mbox{\qedhere} $$
\end{proof}

\begin{lemma}
 \label{injective_const_ext}
 The constant extension restricts to an injective map 
 $$ p^\ast:\ \PBundec X \ \amalg \ \PBuntr X \ \longhookrightarrow \ \PBundec X'\;. $$
\end{lemma}

\begin{proof}
 Since $p^\ast p_\ast(\cL)\simeq\cL\oplus\cL^\sigma$ for a line bundle $\cL$ over $X'$, it is clear that the image is contained in $\PBundec X'$. The images of $\PBundec X$ and $\PBuntr X$ are disjoint since elements of the image of the latter set decompose into line bundles over $X'$ that are not defined over $X$. If we denote taking the inverse elements by $\inv$, then by Proposition \ref{char_PBundec}, $p^\ast$ is injective restricted to $\PBundec X$ because $(\Cl X/\inv) \to (\Cl X'/\inv)$ is. Regarding $\PBuntr X$, observe that 
 \begin{align*} 
  p^\ast(t_D) \ &= \ p^\ast p_\ast(\cL_D) \\ 
  &\simeq \ \cL_D\oplus\cL_{\sigma D} \\ 
  &\sim \ \cL_{D-\sigma D}\oplus \cO_X \\
  &= \ c_{D-\sigma D} \;, 
 \end{align*}  
 where by Proposition \ref{char_PBuntr}, $D$ represents an element in $\bigl(\Cl X'/\Cl X\bigr)/\inv$, and by Proposition \ref{char_PBundec}, $D-\sigma D$ represents an element in $\Cl X'/\inv$. If there are $[D_1], [D_2] \in \Cl X'$ such that $(D_1-\sigma D_1)=\pm(D_2-\sigma D_2)$, then we have $D_1\mp D_2 = \sigma(D_1\mp D_2)$, and consequently $[D_1\mp D_2] \in \Cl X$.
\end{proof}

\begin{rem}
 The constant extension also restricts to a map 
 $$ p^\ast: \ \PBungi X \ \longrightarrow \ \PBungi X' \;. $$
 But this restriction is in general not injective in contrast to the previous result. For a counterexample to injectivity, see \cite[Rem.\ 2.7]{Lorscheid3}.
\end{rem}


\section{Reduction theory for rank $2$ bundles}
 \label{section_reduction_theory}

\noindent
 In this section, we introduce reduction theory for rank $2$ bundles, i.e.\ an invariant $\delta$ which is closely related to the slope of a vector bundle and reduction theory. Namely, a rank $2$ bundle $\cM$ is (semi) stable if and only if $\delta(\cM)$ is negative (non-positive). The invariant $\delta$ is also defined for projective line bundles and will be help to determine the structure of the graphs $\cG_x$.

\begin{pg}
 \label{intro_reduction}
 Vector bundles do not form a full subcategory of the category of sheaves, to wit, if $\cM_1$ and $\cM_2$ are vector bundles and $\cM_1 \to \cM_2$ is a morphism of sheaves, then the cokernel may have nontrivial torsion, which does not occur for a morphism of vector bundles. Thus by a {\it line subbundle $\cL\to\cM$} of a vector bundle $\cM$, we mean an injective morphism of sheaves such that the cokernel $\cM/\cL$ is again a vector bundle.
 
 But every locally free subsheaf $\cL\to\cM$ of rank $1$ extends to a uniquely determined line subbundle $\overline \cL\to\cM$, viz.\ $\overline \cL$ is determined by the constraint $\cL\subset\overline\cL$ (\cite[p.\ 100]{Serre}). On the other hand, every rank $2$ bundle has a line subbundle (\cite[Corollary V.2.7]{Hartshorne}).

 Two line subbundles $\cL\to\cM$ and $\cL'\to\cM$ are said to be the same if their images coincide, or, in other words, if there is an isomorphism $\cL\simeq\cL'$ that commutes with the inclusions into $\cM$. 

 For a line subbundle $\cL\to\cM$ of a rank $2$ bundle $\cM$, we define
 $$ \delta(\cL,\cM) \ := \ \deg \cL - \deg (\cM/\cL) \ = \ 2 \deg \cL - \deg \cM $$
 and
 $$ \delta(\cM) \ := \sup_{\begin{subarray}{c}\cL\to\cM\\ \textrm{line subbundle}\end{subarray}} \delta(\cL,\cM) \;. $$
 If $\delta(\cM)=\delta(\cL,\cM)$, then we call $\cL$ a {\it line subbundle of maximal degree}, or briefly, a {\it maximal subbundle}. Since $\delta(\cL\otimes\cL',\cM\otimes\cL') = \delta(\cL,\cM)$ for a line bundle $\cL'$, the invariant $\delta$ is well-defined on $\PBun_2 X$, and we put $\delta([\cM])=\delta(\cM)$.
 
 Let $g_X$ be the genus of $X$. Then the Riemann-Roch theorem and Serre duality imply:
\end{pg}

\begin{prop}[{\cite[II.2.2, Prop.\ 6 and 7]{Serre}}]
 \label{range_delta}
 For every rank $2$ bundle $\cM$,
 $$ -2g_X\leq\delta(\cM)<\infty \;. $$ 
 If $\cL\to\cM$ is a line subbundle with $\delta(\cL,\cM) > 2g_X-2$, then $\cM \simeq \cL \oplus \cM/\cL$.
\end{prop}

\begin{pg}
 \label{PExt_to_Bun}
 Every extension of a line bundle $\cL'$ by a line bundle $\cL$, i.e.\ every exact sequence of the form
 $$ \ses{\cL}{\cM}{\cL'} \;, $$
 determines a rank $2$ bundle $\cM\in\Bun_2 X$. This defines for all $\cL,\cL'\in\Pic X$ a map
 $$ \Ext^1(\cL,\cL') \ \longrightarrow \ \Bun_2 X \;, $$
 which maps the zero element to $\cL \oplus \cL'$. Remark that since decomposable bundles may have line subbundles that differ from its given two factors, nontrivial elements can give rise to decomposable bundles.

 The units $\FF_q^\times$ operate by multiplication on the $\FF_q$-vector space
 $$ \Ext^1(\cL,\cL') \ \underset{\begin{subarray}{c}{\rm Serre}\\{\rm duality}\end{subarray}}{\simeq} \ \Hom(\cL,\cL'\omega_X^\vee) $$
 where $\omega_X$ is the canonical sheaf of $X$. The multiplication of a morphism $\cL\to\cL'\omega_X^\vee$ by an $a\in\FF_q^\times$ is nothing else but multiplying the stalk $(\cL)_\eta$ by $a^{-1}$ and all stalks $(\cL'\omega_X^\vee)_x$ at closed points $x$ by $a$, which induces automorphisms on both $\cL$ and $\cL'\omega_X^\vee$, respectively. Thus, two elements of $\Ext^1(\cL,\cL')$ that are $\FF_q^\times$-multiples of each other define the same bundle $\cM\in\Bun_2 X$. We get a well-defined map 
 $$ \PP\Ext^1(\cL,\cL') \ \longrightarrow \ \Bun_2 X $$
 where the projective space $\PP\Ext^1(\cL,\cL')$ is defined as the empty set when $\Ext^1(\cL,\cL')$ is trivial. 
 If we further project to $\PBun_2 X$, we can reformulate the above properties of the invariant $\delta$ as follows.
\end{pg}

\begin{prop}
 The map
 $$ \coprod_{-2g_X\leq\deg\cL\leq2g_X-2} \PP\Ext^1(\cL,\cO_X) \ \longrightarrow \ \PBun_2 X $$
 meets every element of $\PBunindec X$, and the fibre of any $[\cM]\in\PBun_2 X$ is of the form
 $$ \left\{ \ 0\to\cL\to\cM\to\cO_X\to0 \ \left| \ \begin{subarray}{c}
    \delta(\cL,\cM)\geq-2g_X\vspace{0,1cm}\\ \text{and }\cM\smallnsimeq\cL\oplus\cO_X\end{subarray} \ \right.\right\} \;. $$
 
\end{prop}

\begin{proof}
 We know that every $[\cM]\in\PBun_2 X$ has a reduction
 $$ \ses{\cL}{\cM}{\cL'} $$
 with $\delta(\cL,\cM)\geq -2g_X$, where we may assume that $\cL'=\cO_X$ by replacing $\cM$ with $\cM\otimes(\cL')^{-1}$, hence
 $\delta(\cL,\cM)=\deg\cL$.
 If $\deg\cL>2g_X-2$, then $\cM$ decomposes, so $\Ext^1(\cL,\cO_X)$ is trivial and $\PP\Ext^1(\cL,\cO_X)$ is
 the empty set. This explains the form of the fibres and that $\PBunindec X$ is contained in the image. 
\end{proof}

\begin{cor}
 \label{finite_indec_part}
 There are only finitely many isomorphism classes of indecomposable projective line bundles.
\end{cor}

\begin{proof}
 This is clear since 
 $ \coprod\limits_{-2g_X\leq\deg\cL\leq2g_X-2} \hspace{-0,7cm} \PP\Ext^1(\cL,\cO_X) \ \ \ $
 is a finite union of finite sets.
\end{proof}

\begin{lemma}
 \label{lemma_schleich}
 If $\cL\to\cM$ is a maximal subbundle, then for every line subbundle $\cL'\to\cM$ that is different from $\cL\to\cM$, 
 $$ \delta(\cL',\cM) \leq -\,\delta(\cL,\cM) \;. $$
 Equality holds if and only if $\cM\simeq\cL\oplus\cL'$, i.e.\ $\cM$ decomposes and $\cL'$ is a complement of $\cL$ in $\cM$. 
\end{lemma}

\begin{proof}
 Compare with \cite[Lemma 3.1.1.]{Schleich}. Since $\cL'\to\cM$ is different from $\cL\to\cM$, there is no inclusion $\cL'\to\cL$ that commutes with the inclusions into $\cM$. Hence the composed morphism $\cL'\to\cM\to\rquot{\cM}{\cL}$ must be injective, and $\deg\cL'\leq\deg\rquot{\cM}{\cL}=\deg\cM-\deg\cL$. This implies that $\delta(\cL',\cM)=2\deg\cL'-\deg\cM\leq\deg\cM-2\deg\cL=-\delta(\cL,\cM)$. Equality holds if and only if $\cL'\to\rquot{\cM}{\cL'}$ is an isomorphism, and in this case, its inverse defines a section $\rquot{\cM}{\cL}\simeq\cL'\to\cM$.
\end{proof}

\begin{prop}\ 
 \label{number_maximal_subbundles}
 \begin{enumerate}
  \item\label{max1} A rank $2$ bundle $\cM$ has at most one line subbundle $\cL\to\cM$ such that $\delta(\cL,\cM)\geq 1$.
  \item\label{max2} If $\cL\to\cM$ is a line subbundle with $\delta(\cL,\cM)\geq0$, then $\delta(\cM)=\delta(\cL,\cM)$.
  \item\label{max3} If $\delta(\cM) = 0$, we distinguish three cases.
        \begin{enumerate}
         \item\label{max31} $\cM$ has only one maximal line bundle: this happens if and only if $\cM$ is indecomposable or if $\cM\simeq\cL_1\oplus\cL_2$ and $\deg\cL_1\neq\deg\cL_2$.
         \item\label{max32} $\cM$ has exactly two maximal subbundles $\cL_1\to\cM$ and $\cL_2\to\cM$: this happens if and only if $\cM\simeq\cL_1\oplus\cL_2$ and $\deg\cL_1=\deg\cL_2$, but $\cL_1\nsimeq\cL_2$.
         \item\label{max33} $\cM$ has exactly $q+1$ maximal subbundles: this happens if and only if all maximal subbundles are of the same isomorphism type $\cL$ and $\cM\simeq\cL\oplus\cL$.
        \end{enumerate}
  \item\label{max4} $\delta(c_D)=\norm{\deg D}$.
  \item\label{max5} $\delta(\cM)$ is invariant under extension of constants for $[\cM]\in\PBundec X$.
  \end{enumerate}
\end{prop} 

\begin{proof}
 Everything follows from the preceding lemmas, except for the fact that $\cL\oplus\cL$ has precisely $q+1$ maximal subbundles in part \eqref{max33}, which needs some explanation. 
 
 First observe that by tensoring with $\cL^{-1}$, we reduce the question to searching the maximal subbundles of $\cO_X\oplus\cO_X$. This bundle has a canonical base at every stalk and the canonical inclusions $\cO_{X,x}^2\hookrightarrow\cO_{X,\eta}^2$ of the stalks at closed points $x$ into the stalk at the generic point $\eta$. This allows us to choose for any line subbundle $\cF\to\cO_X\oplus\cO_X$ a trivialisation with trivial coordinate changes. Thus for every open subset over which $\cF$ trivialises, we obtain the same $1$-dimensional $F$-subspace $\cF_\eta\subset\cO_{X,\eta}^2=F^2$. On the other hand, every $1$-dimensional subspace $\cF_\eta\subset\cO_{X,\eta}^2$ gives back the line subbundle by the inclusion of stalks $\cF_x=\cF_\eta\cap\cO_{X,x}^2\hookrightarrow\cF_\eta$. We see that for every place $x$, $\deg_x \cF\geq 0$, and only the lines in $\cO_{X,\eta}^2=F^2$ that are generated by an element in $\FF_q^2\subset F^2$ define line subbundles $\cF\to\cO_X\oplus\cO_X$ with $\deg_x\cF=0$ for every place $x$. But there are $q+1=\#\PP^1(\FF_q)$ different such line subbundles.
\end{proof}

\begin{prop}
 \label{delta_traces}
 Let $p:X'=X\otimes\FF_{q^2}\to X$ and $\cL\in\Pic X'$, then $\delta(p_\ast\cL)$ is an even non-positive integer. It equals $0$ if and only if $\cL\in p^\ast\Pic X$.
\end{prop}

\begin{proof}
 Over $X'$, we have $p^\ast p_\ast\cL \simeq \cL\oplus\cL^\sigma$, and $\deg\cL = \deg\cL^\sigma$, thus by the previous paragraph, a maximal subbundle of $p_\ast\cL$  has at most the same degree as $\cL$, or, equivalently, $\delta(p_\ast\cL)\leq0$. A maximal subbundle has the same degree as $\cL$ if and only if it is isomorphic to $\cL$ or $\cL^\sigma$ which can only be the case when $\cL$ already is defined over $X$. Finally, by the very definition of $\delta(\cM)$ for rank $2$ bundles $\cM$, it follows that
 $$ \delta(\cM) \ \equiv \ \deg \cM \ \pmod 2 \;, $$
 and $\deg(p_\ast\cL)=2\deg\cL$ is even.
\end{proof}

\begin{rem}
 We see that for $[\cM]\in\PBuntr X$, the invariant $\delta(\cM)$ must get larger if we extend constants to $\FF_{q^2}$, because $p^\ast(\cM)$ decomposes over $X'$. This stays in contrast to the result for classes in $\PBundec X$ (Proposition  \ref{number_maximal_subbundles} \eqref{max5}).
\end{rem}


\section{Nucleus and cusps}
\label{section_nucleus_and_cusps}

\noindent
 In this section, we will define certain subgraphs of $\cG_x$ for a place $x$, namely, the cusp of a divisor class modulo $x$, which is an infinite subgraph of a simple nature, and the nucleus, which is a finite subgraph that depends heavily on the arithmetic of $F$. Finally, $\cG_x$ can be described as the union of the nucleus with a finite number of cusps.
 
\begin{pg}
 \label{def_I_x}
 We use reduction theory to investigate sequences of the form
 $$ \ses{\cM'}{\cM}{\cK_x} \;, $$
 which occur in the definition of $\cU_x([\cM])$. By additivity of the degree map (paragraph \ref{det_and_deg}), $\deg\cM' = \deg\cM-d_x$ where $d_x$ is the degree of $x$.

 If $\cL\to\cM$ is a line subbundle, then we say that it lifts to $\cM'$ if there exists a morphism $\cL\to\cM'$ such that the diagram
 $$  \xymatrix{&{\cL}\ar[dl]\ar{[d]}\\ {\cM'}\ar[r]&{\cM}} $$
 commutes. In this case, $\cL\to\cM'$ is indeed a subbundle since otherwise it would extend non-trivially to a subbundle $\overline\cL\to\cM'\subset\cM$ and would contradict the hypothesis that $\cL$ is a subbundle of $\cM$. By exactness of the above sequence, a line subbundle $\cL\to\cM$ lifts to $\cM'$ if and only if the image of $\cL$ in $\cK_x$ is $0$. 

 Let $\cI_x\subset\cO_X$ be the kernel of $\cO_X\to\cK_x$. This is also a line bundle, since $\cK_x$ is a torsion sheaf. For every line bundle $\cL$, we may think of $\cL\cI_x$ as a subsheaf of $\cL$. In $\Pic X$, the line bundle $\cI_x$ represents the inverse of $\cL_x$, the line bundle associated to the divisor $x$. In particular, $\deg \cI_x=\deg \cL_x^{-1}=-d_x$.

 If $\cL\to\cM$ does not lift to a subbundle of $\cM'$, we have that $\cL\cI_x\subset\cL\to\cM$ lifts to a subbundle of $\cM'$:
 $$ \xymatrix{{\cI_x\cL}\ar{[d]}\ar@{}|\subset[r]&{\cL}\ar{[d]}\\ {\cM'}\ar[r]&{\cM}\;.} $$

 Note that every subbundle $\cL\to\cM'$ is a locally free subsheaf $\cL\to\cM$, which extends to a subbundle $\overline\cL\to\cM$. If thus $\cL\to\cM$ is a maximal subbundle that lifts to a subbundle $\cL\to\cM'$, then $\cL\to\cM'$ is a maximal subbundle. If, however, $\cL\to\cM$ is a maximal subbundle that does not lift to a subbundle $\cL\to\cM'$, then $\cL\cI_x\to\cM'$ is a subbundle, which is not necessarily maximal. These considerations imply that
$$ \begin{array}{cccccccc} 
  \delta(\cM')  & \!\! \leq \! \!  & 2\deg\cL - \deg\cM'  & \! \!  = \! \!  & 2\deg\cL - (\deg\cM - d_x)  & \! \!  = \! \!  & \delta(\cM) + d_x &\textrm{and} \\
  \delta(\cM')& \! \!  \geq \! \!  & 2\deg\cI_x\cL - \deg\cM' \hspace{0pt} & \! \! = \! \!  & \ 2\deg\cL - 2d_x - (\deg\cM - d_x) \hspace{0pt} & \! \!  = \! \!  & \ \delta(\cM) - d_x \;. 
 \end{array}					 $$
 Since $\delta(\cM') \equiv \deg\cM' = \deg\cM - d_x \pmod 2$, we derive:
\end{pg}

\begin{lemma} 
 \label{degree_of_nb}
 If $\quad 0\to\cM'\to\cM\to\cK_x\to 0\quad$ is exact, then
 $$ \delta(\cM') \ \ \ \in \ \ \ \bigl\{ \ \delta(\cM)-d_x,\ \ \  \delta(\cM)-d_x+2,\ldots,\ \delta(\cM)+d_x \ \bigr\} \;. \mbox{\qed} $$
\end{lemma}

\begin{pg}
 \label{def_assoc_sequence}
 Every line subbundle $\cL\to\cM$ defines a line $\cL/\cL\cI_x$ in 
 $\PP^1\bigl(\cM/(\cM\otimes\cI_x)\bigr)$. By the bijection
 $$ \begin{array}{ccc}
     \left\{\begin{subarray}{c}\text{isomorphism classes of exact}\vspace{0,1cm}\\
                        0\to\cM'\to\cM\to\cK_x\to0 \vspace{0,1cm}\\
                        \text{with fixed }\cM 
     \end{subarray}\right\}     & \overset{1:1}\longrightarrow & \PP^1\bigl(\rquot{\cM}{(\cM\otimes\cI_x)}\bigr) \vspace{0,2cm} \;,\\
     \bigl(0\to\cM'\to\cM\to\cK_x\to0\bigr) & \longmapsto & \cM'/(\cM\otimes\cI_x)
    \end{array} $$
 (cf.\ paragraph \ref{geom_nb}) there is a unique 
 $$ \ses{\cM'}{\cM}{\cK_x} \;, $$
 up to isomorphism with fixed $\cM$, such that $\cL\to\cM$ lifts to
 $\cL\to\cM'$. 
 We call this the {\it sequence associated to $\cL\to\cM$ relative to $\Phi_x$},
 or for short the {\it associated sequence}, 
 and $[\cM']$ the {\it associated $\Phi_x$-neighbour}. 
 It follows that $\delta(\cM')\geq\delta(\cL,\cM)+d_x$.

 We summarise this.
\end{pg}

\begin{lemma} 
 \label{assoc_sequence}
 If $\cL\to\cM$ is a maximal subbundle, then the associated $\Phi_x$-neighbour $[\cM']$ has $\delta(\cM')=\delta(\cM)+d_x$, and
 $$ \sum_{\begin{subarray}{c} ([\cM],[\cM'],m)\in\,\cU_x([\cM])\vspace{0,1cm}\\ \delta(\cM')=\delta(\cM)+d_x \end{subarray}}\hspace{-0,6cm} m 
	 \hspace{0,4cm} = \hspace{0,2cm} \# \ \left\{\ \overline\cL\in\PP^1\bigl(\rquot{\cM}{(\cM\otimes\cI_x)}\bigr) \ \left| \ 
    \begin{subarray}{c}\exists\cL\to\cM\textrm{ maximal subbundle} \\ \textrm{with }\cL\equiv\overline\cL\pmod{\cM\otimes\cI_x}\end{subarray} 
    \ \right. \right\} \;. \mbox{\qed} $$
\end{lemma}

\begin{thm}
 \label{edges_of_cusps}
 Let $x$ be a place and $[D]\in\Cl X$ be a divisor of non-negative degree. The $\Phi_x$-neighbours $v$ of $c_D$ with $\delta(v)=\deg D+d_x$ are given by the following list:
 \begin{align*}
  (c_0,c_x,q+1) \ &\in \ \cU_x(c_0), \\
  (c_D,c_{D+x},2) \ &\in \ \cU_x(c_D)\quad\text{if }[D]\in(\Cl^0 X)[2]-\{0\}, \\
  (c_D,c_{D+x},1), (c_D,c_{-D+x},1) \ &\in \ \cU_x(c_D)\quad\text{if }[D]\in\Cl^0 X -(\Cl^0 X)[2],\text{ and} \\
  (c_D,c_{D+x},1) \ &\in \ \cU_x(c_D)\quad\text{if }\deg D\text{ is positive.}  
 \end{align*}
 For all $\Phi_x$-neighbours $v$ of $c_D$ not occurring in this list, $\delta(v)<\delta(c_D)+d_x$.
 If furthermore $\deg D>d_x$, then $\delta(v)=\deg D-d_x$, and if $\deg D>m_X+d_x$ where $m_X = \max\{2g_X-2,0\}$, then
 $$ \cU_x(c_D) \ = \ \{(c_D,c_{D-x},q_x),(c_D,c_{D+x},1)\} \;. $$
\end{thm}

\begin{proof}
 By Lemma \ref{assoc_sequence}, the $\Phi_x$-neighbours $v$ of $c_D$ with $\delta(v)=\delta(c_D)+d_x$ counted with multiplicity correspond to the maximal subbundles of a rank $2$ bundle $\cM$ that represents $c_D$. Since $\delta(\cM)=\delta(c_D)\geq0$, the list of all $\Phi_x$-neighbours $v$ of $c_D$ with $\delta(v)=\deg D+d_x=\delta(c_D)+d_x$ follows from the different cases in Proposition \eqref{number_maximal_subbundles} \eqref{max1} and \eqref{max3}. Be aware that $c_D=c_{-D}$ by Proposition \ref{char_PBundec}; hence it makes a difference whether or not $D$ is $2$-torsion.
 
 For the latter statements, write $\cM=\cL_D\oplus\cO_X$ and let $\cM'$ be a subsheaf of $\cM$ with cokernel $\cK_x$ such that $\delta(\cM')<\delta(\cM)+d_x$. Then $\cL_D\to\cM$ does not lift to $\cM'$, but $\cL_D\cI_x\to\cM'$ is a line subbundle and 
 $$ \cM'/\cL_D\cI_x \ \simeq \ (\det\cM')(\cL_D\cI_x)^\vee \ \simeq \ (\det\cM)\cI_x(\cL_D\cI_x)^\vee \ \simeq \ \cL_D\cI_x(\cL_D\cI_x)^\vee \ \simeq \ \cO_X \;. $$ 
 
 If $\deg D>d_x$, then
 $$\delta(\cL_D\cI_x,\cM') \ = \ \deg \cL_D\cI_x - \deg \cO_X \ = \ \deg D-d_x \ > \ 0\;.$$
 Proposition \ref{number_maximal_subbundles} \eqref{max1} implies that $\cL_D\to\cM$ is the unique maximal subbundle of $\cM'$ and thus $\delta(\cM')=\delta(\cM)-d_x$. 

 If $\delta(\cM)>m_X+d_x$, then $\delta(\cM')>m_X\geq2g_X-2$, hence $\cM'$ decomposes and represents $c_{D-x}$. Since the multiplicities of all $\Phi_x$-neighbours of a vertex sum up to $q_x+1$, this proves the last part of our assertions.
\end{proof}

\begin{df}
 \label{def_cusps_and_nucleus}
 Let $x$ be a place. Let the divisor $D$ represent a class $[D]\in\Cl \cO_X^x= \Cl X\,/\langle x\rangle$. We define the {\it cusp $\cC_x(D)$ (of $D$ in $\cG_x$)} as the full subgraph of $\cG_x$ with vertices
 $$ \Vertex \cC_x(D) \ = \ \bigl\{\; c_{D'}\,\bigl|\,[D']\equiv [D]\pmod{\langle x\rangle},\text{ and }\deg D'>m_X\,\bigr.\bigr\} \;, $$
 and the {\it nucleus $\cN_x$ (of $\cG_x$)} as the full subgraph of $\cG_x$ with vertices
 $$ \Vertex \cN_x \ = \ \bigl\{\, [\cM]\in\PBun_2 X \,\bigl|\, \delta(\cM) \leq m_X+d_x \,\bigr.\bigr\} \;. $$
\end{df}

\begin{pg}
 Theorem \ref{edges_of_cusps} determines all edges of a cusp $\cC_x(D)$. If $m_X<\deg D\leq m_X+d_x$, the cusp can be illustrated as in Figure \ref{figure_cusp}. Note that a cusp is an infinite graph. It has a regular pattern that repeats periodically. In diagrams we draw the pattern and indicate its periodic continuation with dots.
\end{pg}

\begin{figure}[htb]
 \begin{center}
  \includegraphics{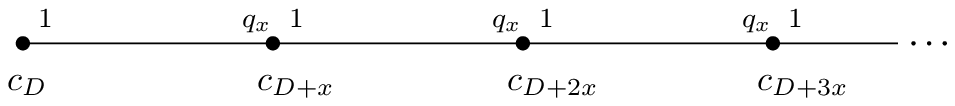}
  \caption{A cusp}
  \label{figure_cusp}
 \end{center} 
\end{figure}

 We summarise the theory so far in the following theorem that describes the general structure of $\cG_x$.

\begin{thm}
 \label{gen_str_thm}
 Let $x$ be a place of degree $d_x$ and $h_X=\#\Cl^0X$ be the class number.
 \begin{enumerate}
  \item \label{genstrthm1} $\cG_x$ has $h_Xd_x$ cusps and
        $$ \cG_x \ = \ \cN_x \cup \coprod_{[D]\in\Cl \cO_F^x} \cC_x(D) \;, $$
        where $\Vertex\cN_x\cap\Vertex\cC_x(D)=\{c_D\}$ if $m_X<\deg D\leq m_X+d_x$. The union of the edges is disjoint. Different cups are disjoint subgraphs.
  \item \label{genstrthm2} $\cN_x$ is finite and has $\#\bigl(\Cl \cO_F^x\,/\,2\Cl \cO_F^x\bigr)$ components. Each vertex of $\cN_x$ is at distance $\leq(2g_X+m_X+d_x)/d_x$ from some cusp. The associated CW-complexes of $\cN_x$ and $\cG_x$ are homotopy equivalent.
  \item \label{genstrthm3} If $[D]\in\Cl\cO_F^x$, then $\Vertex \cC_x(D)\subset \PBundec X$. Furthermore
        \begin{align*}
		  \PBundec X \ &\subset \ \{v\in\Vertex\cG_x\mid\delta(v)\geq0\}\;, \\
		  \PBungi X  \ &\subset \ \{v\in\Vertex\cG_x\mid\delta(v)\leq2g-2\} \text{ and} \\
		  \PBuntr X  \ &\subset \ \{v\in\Vertex\cG_x\mid\delta(v)<0\text{ and even}\} \;. 
	\end{align*}
 \end{enumerate}
\end{thm}

\begin{figure}[h]
 \begin{center}
  \includegraphics[width=1.0\textwidth,angle=0]{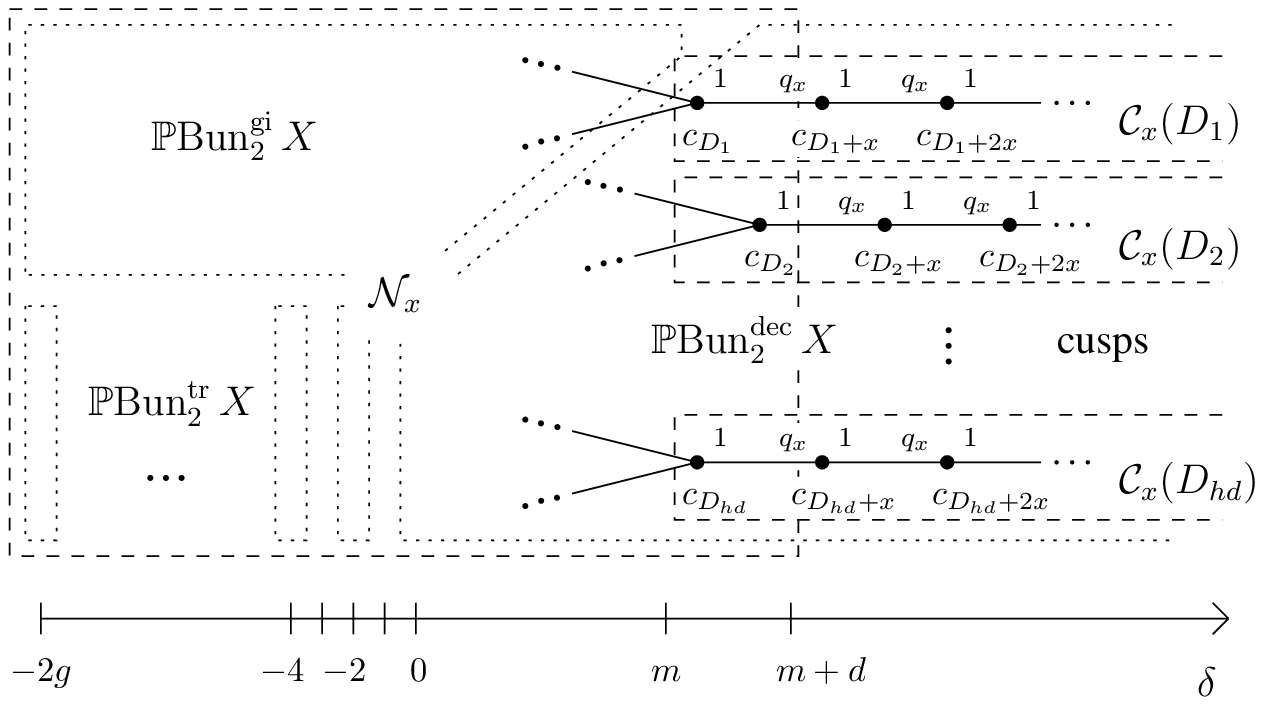}
  \caption{General structure of $\cG_x$} 
  \label{figure_genstr}
 \end{center} 
\end{figure}

\begin{pg}[Remark on Figure \ref{figure_genstr}]
 Define $h=h_X$, $m=m_X$, $d=d_x$ and $q_x=q^{\deg x}$. Further let $D_1,\dotsc,D_{hd}$ be representatives for $\Cl\cO_F^x$ with $m<\deg D_i\leq m+d$ for $i=1\dotsc,hd$. The cusps $\cC_x(D_i)$, $i=1,\dotsc,hd$, can be seen in Figure \ref{figure_genstr} as the subgraphs in the dashed regions that are open to the right. The nucleus $\cN_x$ is contained in the dotted rectangle to the left. Since we have no further information about the nucleus, we leave the area in the rectangle open.

 The $\delta$-line on the bottom of the picture indicates the value $\delta(v)$ for the vertices $v$ in the graph that lie vertically above $\delta(v)$.

 The dotted regions refer to the sort of vertices, which are elements of either $\PBungi X$, $\PBuntr X$, or $\PBundec X$. All lines are drawn with reference to the $\delta$-line to reflect part \eqref{genstrthm3} of the theorem.
\end{pg}

\begin{proof}
 The number of cusps is $\#\Cl\cO_X^x \,=\,\#(\rquot{\Cl X}{\langle x\rangle}) \,=\, \#\Cl^0X\cdot\#(\ZZ/d_x\ZZ)\,=\,h_Xd_x$. That the vertices of cusps are disjoint and only intersect in the given point with the nucleus, is clear by definition. Regarding the edges, recall from paragraph \ref{Psi_x} that if there is an edge from $v$ to $w$ in $\cG_x$, then there is also an edge from $w$ to $v$. But Theorem \ref{edges_of_cusps} implies that each vertex of a cusp that does not lie in the nucleus only connects to a vertex of the same cusp, hence every edge of $\cG_x$ either lies in a cusp or in the nucleus. Different cusps are disjoint by definition. This shows \eqref{genstrthm1}.
 
 The nucleus is finite since $\PBunindec X$ is finite by Corollary \ref{finite_indec_part} and since the intersection $\PBundec X\cap\Vertex\cN_x$ is finite by the definition of the nucleus and Proposition \ref{char_PBundec}. Since the cusps are contractible as CW-complexes, $\cN_x$ and $\cG_x$ have the same homotopy type. Therefore $\cN_x$ has $\#\bigl(\Cl \cO_F^x\,/\,2\cO_F^x\bigr)$ components by Proposition \ref{strong_approx+comp}. By Lemma \ref{assoc_sequence}, every vertex $v$ has a $\Phi_x$-neighbour $w$ with $\delta(w)=\delta(v)+d_x$, thus the upper bound for the distance of vertices in the nucleus to one of the cusps. This proves \eqref{genstrthm2}.
 
 The four statements of Part \eqref{genstrthm3} follow from the definition of a cusp, Proposition \mbox{\ref{number_maximal_subbundles} \eqref{max4}}, Proposition \ref{range_delta} and Proposition \ref{delta_traces}, respectively.
\end{proof}

\begin{eg}[The projective line]
 \label{eg_P1}
 Let $X$ be the projective line over $\FF_q$. Then $g_X=0$, $h_X=1$ and $X$ has a closed point $x$ of degree $1$.
 This means that
 $$ \PBundec X \ = \ \{c_{nx}\}_{n\geq0} \;. $$
 Since an indecomposable bundle $\cM$ must satisfy both $\delta(\cM)\geq0$ and $\delta(\cM)\leq-2$ which is impossible,
 all projective line bundles decompose.
 Theorem \ref{edges_of_cusps} together with the fact that the weights around each vertex sum to $q+1$ in the graph of $\Phi_x$
 determines $\cG_x$ completely, as illustrated in Figure \ref{P1-graph}. 
\end{eg}

\begin{figure}[htb]
 \begin{center}
  \includegraphics{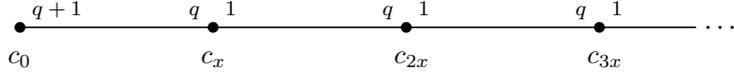}
  \caption{The graph of $\Phi_x$ for a degree one place $x$ of a rational function field}
  \label{P1-graph}
 \end{center} 
\end{figure}


\section{Application to automorphic forms}
\label{section_automorphic_forms}

\noindent
 In this section, we explain how to recover automorphic forms as functions on the graph and indicate how unramified automorphic forms can be explicitly calculated as functions on the graph by solving a finite system of linear equations. We begin with recalling the definition of an automorphic form.

\begin{pg}
 A function $f\in C^0(G_\AA)$ is called an \emph{automorphic form (for $\PGL_2$ over $F$)} if there is a compact open subgroup $K'$ of $G_\AA$ such that $f$ is left $G_F$-invariant and right $K'$-invariant and if it generates a finite-dimensional $\cH_{K'}$-subrepresentation $\cH_{K'}(f)$ of $C^0(G_\AA)$. We denote the space of automorphic forms by $\cA$ and note that the action of $\cH$ on $C^0(G_\AA)$ restricts to $\cA$. We denote the subspace of right $K'$-invariant automorphic forms by $\cA^{K'}$, a space on which $\cH_{K'}$ acts. We can reinterpret the elements in $\cA^{K'}$ as functions on $\lrquot{G_F}{G_\AA}{K'}$, which is the vertex set of the graph $\cG_{\Phi,K'}$ of a Hecke operator $\Phi\in\cH_{K'}$. 

 We shall investigate the space $\cA^K$ of unramified automorphic forms in more detail. We write $f(v)$ or $f(\cM)$ for the value $f(g)$ if $v=[g]$ is the class of $g$ in $\lrquot{G_F}{G_\AA}{K}$ and $\cM=\cM_g$ is the rank $2$ bundle that corresponds to $g$. In particular, we can see $f$ also as a function on $\PBun_2 X$.

 The space of automorphic forms decomposes into a cuspidal part $\cA_0$, a part $\cE$ that is generated by derivatives of Eisenstein series and a part $\cR$ that is generated by derivatives of residues of Eisenstein series (for complete definitions, cf.\ \cite{Lorscheid1}). The decomposition decents to unramified automorphic forms: $\cA^K=\cA_0^K\oplus\cE^K\oplus\cR^K$. We describe functions in these parts separately.
\end{pg}

\begin{pg}
 We start with some considerations for $\Phi_x$-eigenfunctions as functions on a cusp $\cC_x(D)$ where $D$ is a divisor with $m_X<\deg D\leq m_X+d_x$:
 \begin{center} \includegraphics{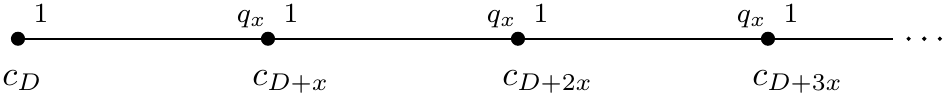} \end{center} 
 Let $f\in\cA^K$ satisfy the eigenvalue equation $\Phi_x f=\lambda f$, then we obtain for every $i\geq1$,
 \begin{equation}\label{eq_cusp_eigenvalue} f(c_{D+(i+1)x})\ =\ \lambda\, f(c_{D+ix})\,-\,q_x\,f(c_{D+(i-1)x}) \;.  \end{equation}
 Thus the restriction of $f$ to $\Vertex\cC_x(D)$ is determined by the eigenvalue $\lambda$ once its values at $c_D$ and $c_{D+x}$ are given. This consideration justifies that we only have to evaluate the eigenvalue equation at vertices of the nucleus to determine the eigenfunctions of $\Phi_x$. 
\end{pg}

\begin{pg}
 The space $\cA_0^K$ has a basis of $\cH_K$-eigenfunctions and every unramified cusp form has a compact, i.e.\ finite, support in $\lrquot{G_F}{G_\AA}{K}$. By the eigenvalue equation \eqref{eq_cusp_eigenvalue} it follows that an Hecke-eigenfunction $f\in\cA_0^K$ must vanish on all vertices of a cusp in order to have compact support. Thus the support of a cusp form is contained in the finite set $V$ of vertices $v$ with $\delta(v)\leq m_X$, and $\cA_0^K$ can be determined by considering a finite number of eigenvalue equations for $\Phi_x$.

 These eigenvalue equations can be described in terms of the matrix $M_x$ associated to $\Phi_x$ (cf.\ paragraph \ref{associated_matrix}). Namely, $\cA_0^K$ is generated by the eigenfunctions of $M_x$ whose support is contained in $V$. This problem can be rephrased into a question on the finite submatrix $M'_x=(a_{v,w})_{v\in V,w\in\Vertex\cN_x}$ of $M_x=(a_{v,w})_{v,w\in\Vertex\cG_x}$, which we forgo to spell out.

 In \cite{Moreno} one finds a finite set $S$ of places such that an $\cH_K$-eigenfunction $f\in\cA_0^K$ is already characterised (up to multiple) by its $\Phi_x$-eigenvalues for $x\in S$. This means that one finds the cuspidal $\cH_K$-eigenfunctions by considering the eigenvalue equations for the finitely many vertices $v\in V$ and the finitely many Hecke operators $\Phi_x$ for $x\in S$.
\end{pg}

\begin{pg}
 We proceed with $\cE^K\oplus\cR^K$. This space decomposes into a direct sum of generalised (infinite-dimensional) Hecke-eigenspaces $\cE(\chi)$ where $\chi$ runs through all unramified Hecke characters, i.e.\ continuous group homomorphisms $\chi:\lrquot{F^\times}{\AA^\times}{\cO_\AA^\times}\to\CC^\times$ modulo inversion; in particular, $\cE(\chi)=\cE(\chi^{-1})$. The generalised eigenspace $\cE(\chi)$ is characterised by its unique Hecke-eigenfunction $\varE(\blanc,\chi)$ (up to scalar multiple), which in turn is determined by its $\Phi_x$-eigenvalues $\lambda_x(\chi)=q_x^{1/2}\bigl(\chi(\pi_x)+\chi^{-1}(\pi_x)\bigr)$ for $x\in\norm X$. We have $\cE(\chi)\subset\cE$ if and only if $\chi^2\neq\norm\ ^{\pm1}$, in which case $\varE(\blanc,\chi)$ is an Eisenstein series. For $\chi^2=\norm\ ^{\pm1}$, $\varE(\blanc,\chi)$ is a residue of an Eisenstein series. For details, see \cite{Lorscheid1}; in particular, Theorem 11.10.

 We say that a subset $S\subset\norm X$ generates $\Cl X$ if the classes of the prime divisors corresponding to the places in $S$ generate $\Cl X$. Let $S$ be a set of places that generates $\Cl X$ and satisfies that for every decomposition $S=S_+\cup S_-$ either $2\Cl X=2\langle S_+\rangle$ or $2\Cl X=2\langle S_-\rangle$. This set can be chosen to be finite. Then the Hecke eigenfunction $\varE(\blanc,\chi)$ is uniquely determined (up to scalar multiples) by the $\Phi_x$-eigenvalues $\lambda_x(\chi)$. For details, see \cite[pg.\ 3.7.10]{Lorscheid-thesis}. 

 In order to describe an Eisenstein series or a residue of an Eisenstein series, one only needs to consider the finitely many eigenvalue equations for vertices $v\in V$ for the finitely many Hecke-operators $\Phi_x$ with $x\in S$. Derivatives of Eisenstein series or residues are similarly determined by generalised eigenvalue equations, see \cite[Lemmas 11.2 and 11.7]{Lorscheid1} for the explicit formulas.

 In the case of a residue, i.e.\ $\chi^2=\norm\ ^{\pm1}$, the function $f=\varE(\blanc,\chi)$ has a particular simple form. Namely, $\chi$ is of the form $\omega\norm\ ^{\pm 1/2}$ where $\omega^2=1$, and thus 
 $$ \lambda_x(\chi) \ = \ q_x^{1/2}\bigl(\omega(\pi_x)\norm{\pi_x}^{\pm 1/2}+\omega(\pi_x)\norm{\pi_x}^{\mp 1/2}\bigr) \ = \ \omega(\pi_x) (q_x+1). $$
 Since every vertex $v$ has precisely $(q_x+1)$ $\Phi_x$-neighbors (counted with multiplicities), we have $f(v)=\omega(\pi_x)f(w)$ for all adjacent vertices $v$ and $w$. 
\end{pg}

\begin{rem}
 The methods of this paragraph will be applied in \cite{Lorscheid3} to determine the space of unramified cusp forms for an elliptic function field and to show that there are no unramified toroidal cusp forms in this this case.
\end{rem}


\section{Finite-dimensionality results}
\label{section_finite-dimensionality}

\noindent
 In this section, we will show how the theory of the last sections can be used to show finite dimensionality of subspaces of $C^0(G_\AA)^K$ whose elements $f$ are defined by a condition of the form
 $$ \sum_{i=1}^n m_i \Phi(f)(g_i) \ = \ 0 $$
 for all $\Phi\in\cH_K$ (with $m_i\in\CC$ and $g_i\in G_\AA$ being fixed). We will explain a general technique and apply it to show that the spaces of functions in $C^0(G_\AA)^K$ satisfying the cuspidal condition respective the toroidal condition are finite-dimensional. In particular, this implies that all functions satisfying one of these conditions are automorphic forms.

\begin{pg}
 Write $\PCl X$ for the set of divisor classes that are represented by prime divisors and $\ECl X$ for the semigroup they generate, viz.\ for all classes that are represented by effective divisors. In particular, $\ECl X$ contains $0$, the class of the zero divisor, and for all other $[D]\in\ECl X$, we have $\deg D>0$. Denote by $\Cl^d X$ the set of divisor classes of degree $d$ and by $\Cl^{\geq d} X$ the set of divisor classes of degree at least $d$. Let $g_X$ be the genus of $X$.
\end{pg}

\begin{lemma}
 \label{many_effective_divisors}
 $$ \Cl^{\geq g_X} X \quad\subset\quad \ECl X \;. $$
\end{lemma}

\begin{proof}
 Let $C$ be a canonical divisor on $X$, which is of degree $2g_X-2$. For a divisor $D$, define $l(D)=\dim_{\FF_q}H^0(X,\cL_D)$. We have $[D]\in\ECl X$ if and only if $l(D)>0$, cf.\ \cite[Section IV.1]{Hartshorne}. The Riemann-Roch theorem is 
 $$ l(D)\,-\,l(D-C) \ = \ \deg D\,+\,1\,-\,g_X \;, $$
 cf.\ \cite[Thm.\ IV.1.3]{Hartshorne}. 
 
 If now $[D]\in\Cl^{\geq g_X} X$, then $\deg D\geq g_X$ and the Riemann-Roch theorem implies that $l(D)\geq \deg D+1-g_X>0$.
\end{proof}

\begin{pg}
 Let $D$ be an effective divisor. Then it can be written in a unique way up to permutation of terms as a sum of prime divisors $D=x_1+\ldots+x_n$. We set $\Phi_D=\Phi_{x_1}\dotsm\Phi_{x_n}$. Since $\cH_K$ is commutative, $\Phi_D$ is well-defined. Further we briefly write $\cG_D$ for the graph $\cG_{\Phi_D,K}$ of $\Phi_D$, and $\cU_D(v)$ for $\cU_{\Phi_D,K}(v)$.
 
 Let $[D]\in\Cl X$. Recall from paragraph \ref{det_and_deg} that $\cL_D$ denotes the associated line bundle and from paragraph \ref{labelling_vertices} that $c_D$ denotes the vertex that is represented by \mbox{$\cL_D\oplus\cO_X$}. Recall from Proposition \ref{number_maximal_subbundles} \eqref{max4} that $\delta(c_D)=\norm{\deg D}$ where $\delta$ is defined as in paragraph \ref{intro_reduction}. 
\end{pg}

\begin{lemma}
 \label{neighbours_of_c_0}
 Let $D$ be a non-trivial effective divisor.
 \begin{enumerate}
  \item\label{nb1} Let $v,v'\in\Vertex\cG_D$. If $v'$ is a $\Phi_D$-neighbour of $v$, then $\norm{\delta(v')-\delta(v)}\leq\deg D$.
  \item\label{nb2} Let $[\cM]\in\Vertex\cG_D$. Every maximal subbundle $\cL\to\cM$ lifts to a maximal subbundle $\cL\to\cM'$ of a uniquely determined rank $2$ bundle $\cM'$ such that $[\cM']$ is a $\Phi_D$-neighbor of $[\cM]$ with $\delta(\cM')=\delta(\cM)+\deg D$. Conversely, every maximal subbundle $\cL\to\cM'$ extends to a maximal subbundle $\cL\to\cM$ if $[\cM']$ is a $\Phi_D$-neighbor of $[\cM]$ with $\delta(\cM')=\delta(\cM)+\deg D$.
\end{enumerate}	
\end{lemma}

\begin{proof}
 We do induction on the number of factors in $\Phi_D=\Phi_{x_1}\dotsm\Phi_{x_n}$ with $x_1,\ldots,x_n$ being prime divisors. Put $x=x_n$.
 
 If $n=1$, then $\Phi_D=\Phi_x$. Assertion \eqref{nb1} follows from Proposition \ref{degree_of_nb} and assertion \eqref{nb2} follows from Lemma \ref{assoc_sequence} and Theorem \ref{edges_of_cusps}.
 
 If $n>1$, we can write $\Phi_D=\Phi_{D'}\Phi_x$ for the effective divisor $D'=x_1+\dotsb+x_{n-1}$, which is of positive degree $\deg D'=\deg D-\deg x$. Assume that \eqref{nb1} and \eqref{nb2} hold for $D'$. We prove \eqref{nb1}. Let $v'$ be a $\Phi_D$-neighbour of $v$. As explained in paragraph \ref{algebra_strucure_in_graphs}, there is a $v''$ that is a $\Phi_{D'}$-neighbour of $v$ and a $\Phi_x$-neighbour of $v'$, thus the inductive hypothesis and Proposition \ref{degree_of_nb} imply
 $$ \norm{\delta(v')-\delta(v)} \ \leq \ \norm{\delta(v')-\delta(v'')}+\norm{\delta(v'')-\delta(v)}\ \leq\ \deg D'+\deg x\ =\ \deg D \;. $$
 
 We prove \eqref{nb2}. By the inductive hypothesis, the $\Phi_{D'}$-neighbours $[\cM']$ of $[\cM]$ with $\delta(\cM')-\delta(\cM)=\deg D'$ correspond to the maximal subbundles $\cL'\to\cM$, which lift to maximal subbundles $\cL'\to\cM'$. On the other hand, every maximal subbundle $\cL\to\cM'$ of a $\Phi_{D'}$-neighbours $[\cM']$ of $[\cM]$ with $\delta(\cM')-\delta(\cM)=\deg D'$ is of this form since 
 $$ \delta(\cM) \quad = \quad \delta(\cM') - \deg D' \quad = \quad  \delta(\cL,\cM') - \deg D' \quad = \quad \delta(\cL,\cM), $$
 thus $\cL\to\cM$ must be a maximal subbundle. We now apply Lemma \ref{assoc_sequence} to each of the $\Phi_{D'}$-neighbors $\cM'$ of $\cM$ and obtain the first statement of \eqref{nb2}. The second statement of \eqref{nb2} follows from Theorem \ref{edges_of_cusps}.
\end{proof}

\begin{pg}
 We demonstrate how to use the lemma for to show that the space $\cV_0$ of all unramified functions on $G_\AA$ that satisfy the cuspidal condition is finite-dimensional. Namely, let $N\subset G$ be a unipotent subgroup, then the cuspidal condition for $f\in C^0(G_\AA)^K$ is that 
 $$ \int\limits_{\lquot{N_F}{N_\AA}} \Phi(f)(n)\;dn \quad = \quad 0 $$
 for all $\Phi\in\cH_K$. If $f$ is an automorphic form, then this condition defines a cusp form. A posteriori it will be clear that $\cV_0$ contains only automorphic forms and thus equals the space $\cA_0^K$ of unramified cusp forms.
\end{pg}

\begin{thm}
 The dimension of $\cV_0$ is finite and bounded by
 $$ \dim\cV_0 \ \leq \ \#\{[\cM]\in\PBun_2 X | \delta(\cM)\leq m_X \}. $$
 \end{thm}

\begin{proof}
 Note that there are only finitely many projective line bundles $[\cM]$ with $\delta(\cM)\leq m_X$ since $\PBunindec X$ is finite and $\PBundec X$ has only finitely many classes $[\cM]$ with $\delta(\cM)\leq m_X$. So the finite-dimensionality of $\cV_0$ will follow from the inequality.
 
 We proceed with the proof of the inequality. The geometric equivalent of the cuspidal condition is that
 $$ \sum_{\cM\in\Ext^1(\cO_X,\cO_X)} \Phi(f)(\cM) \quad = \quad 0 $$
 for all $\Phi\in\cH_K$ (cf.\ \cite{Gaitsgory}).
 
 Since $\delta(\cO_X,\cM)=0$ for $\cM\in\Ext^1(\cO_X,\cO_X)$, we have that $\cO_X\to\cM$ is a maximal subbundle by Proposition \ref{number_maximal_subbundles}\eqref{max2}, and only in the case of the trivial extension $\cM\simeq\cO_X\oplus\cO_X$, there are other maximal subbundles, namely, there exist $(q+1)$ different subbundles of the form $\cO_X\to\cM$. Note that in any case $\delta(\cM)=0$.

 Let $D$ be a nontrivial effective divisor. In case $\cM$ is the trivial extension $\cO_X\oplus\cO_X$, the vertex $c_0=[\cM]$ has the unique $\Phi_D$-neighbour $v'=c_D$ with $\delta(v')=\deg D$ which is of multiplicity $q+1$. In case, $\cM$ is a non-trivial extension of $\cO_X$ by itself, the vertex $v=[\cM]$ has a unique $\Phi_D$-neighbour $v'=[\cM']$ with $\delta(v')-\delta(v)=\deg D$, which has a unique maximal subbundle, namely, $\cO_X\to\cM'$.
 
 Thus for every $\cM\in\Ext^1(\cO_X,\cO_X)$ and every $\Phi_D$-neighbor $[\cM']$ of $[\cM]$ with $\delta(\cM')=\deg D$, the maximal subbundles of $\cM'$ are of the form $\cO_X\to\cM'$. Thus if $\deg D > m_X$, then $[\cM']=c_D$ by Proposition \ref{range_delta}.

 We finish the proof of the theorem by showing that every $f\in\cV_0$ is determined by its values in the vertices $v$ with $\delta(v)\leq m_X$. We make an induction on $d=\delta(c_D)$, where $c_D$ varies through all vertices $v$ with $\delta(v)>m_X$.

 Let $d>m_X$. Assume that the values of $f$ in all vertices $v$ with $\delta(v)<d$ are given (which is the case when $d=m_X+1$; thus the initial step). Let $v$ be a vertex with $\delta(v)=d$, then $v=c_D$ for an effective divisor $D$ by Lemma \ref{many_effective_divisors} since $m_X=\max\{0,2g_X-2\}\geq g_X-1$. For the Hecke operator $\Phi_D$, the cuspidal condition reads by the previous argumentation and Lemma \ref{neighbours_of_c_0} as
 $$ (q + q^{e_1}) \cdot f(c_D) + \sum_{\delta(v')<d} a_{v'} f(v') \quad = \quad 0 $$
 for certain $a_{v'}$ and $e_1=\dim\Ext^1(\cO_X,\cO_X)$. Thus $f(v)$ is determined by the values $f(v')$ in vertices $v'$ with $\delta(v')<d$, which proves the theorem.
\end{proof}

\begin{pg}
 While the finite-dimensionality of $\cV_0$ can also be established without the techniques of this paper, we do not know any other method to prove corresponding fact for toroidal functions. For more details on the following definitions, see \cite{Lorscheid1}.

 Choose a basis of $\FF_{q^2}$ over $\FF_q$. This defines an embedding of $E=\FF_{q^2}F$ into the algebra of $2\times 2$-matrices with entries in $F$. The image of $E^\times$ is contained in $\GL_2(F)$ and defines a non-split torus $T'$ of $\GL_2$. The image of $T'$ in $G=\rquot{\GL_2}{Z}$ defines a non-split torus $T$ of $G$. 

 A function $f\in C^0(G_\AA)^K$ is \emph{$E$-toroidal} if for all $\Phi\in\cH_K$,
 $$ \int\limits_{\lquot{T_F}{T_\AA}} f(t)\; dt \quad = \quad 0. $$
 We denote the space of all $E$-toroidal functions $f\in C^0(G_\AA)^K$ by $\cV_\tor$. Note that in \cite{Lorscheid1} one finds a toroidal condition, which is stronger than $E$-toroidality. Namely, $f$ has to be $E'$-toroidal for all separable quadratic algebra extensions $E'$ of $F$. We forgo to recall complete definitions, but remark that the finite-dimensionality of the space of all toroidal $f\in C^0(G_\AA)^K$ follows since it is a subspace of $\cV_\tor$.

 Let $p:X'\to X$ be the map of curves that corresponds to the field extension $E/F$, and let $e=\tinymat 1 {} {} 1 $.
\end{pg}

\begin{thm}
 \label{geometric_toroidal_cond}
 Let $c_T=\vol(\lquot{T_F}{T_\AA})\,/\,\#\bigl(\rquot{\Pic X'}{p^\ast(\Pic X)}\bigr)$. Then for all $f\in C^0(G_\AA)^K$,
 $$ f_T(e) \ \ \ = \ \ \ c_T \ \ \ \cdot\hspace{-0,7cm}\sum_{[\cL]\in\rquot{\Pic X'}{p^\ast(\Pic X)}}\hspace{-0,6cm}f([p_\ast\cL]) \;. $$
\end{thm}

\begin{proof}
 Let $\AA_E$ be the adeles of $E$. To avoid confusion, we write $\AA_F$ for $\AA$. We introduce the following notation. For an $x\in\norm X$ that is inert in $E/F$, we define $\cO_{E,x}:=\cO_{E,y}$, where $y$ is the unique place that lies over $x$. For an $x\in\norm X$ that is split in $E/F$, we define $\cO_{E,x}:=\cO_{E,y_1}\oplus\cO_{E,y_2}$, where $y_1$ and $y_2$ are the two places that lie over $x$. Note that there is no place that ramifies. Let $\cO_{E_x}$ denote the completion of $\cO_{E,x}$. Then $\cO_{E_x}$ is a free module of rank $2$ over $\cO_{F_x}=\cO_x$ for every $x\in\norm X$.
 
 Let $\Theta_E:\AA_E^\times\to\GL_2(\AA_F)$ be the base extension of the embedding $E^\times\to\GL_2(F)$ that defines $T'$, which corresponds to the chosen basis of $E$ over $F$ that is contained in $\FF_{q^2}$. This basis is also a basis of $\cO_{E_x}$ over $\cO_{F_x}$ for every $x\in\norm X$. This shows that $\Theta_E^{-1}(\GL_2(\cO_{\AA_F}))=\cO_{\AA_E}^\times$ and that the diagram
 $$ \xymatrix{ \lrquot{E^\times}{\AA_E^\times}{\cO_{\AA_E}^\times}\ar[rr]^(.6){1:1}\ar[d]^{\Theta_E} && \Pic X'\ar[d]^{p_\ast}\\
	 \lrquot{\GL_2(F)}{\GL_2({\AA_F})}{\GL_2(\cO_{\AA_F})}\ar[rr]^(.6){1:1} && \Bun_2X} $$
 commutes, where the horizontal arrows are the bijections as described in paragraph \ref{bundle_to_GL_n}.

 The action of $\AA_F$ on $\lrquot{E^\times}{\AA_E^\times}{\cO_{\AA_E}^\times}$ and $\lrquot{\GL_2(F)}{\GL_2({\AA_F})}{\GL_2(\cO_{\AA_F})}$ by scalar multiplication is compatible with the action of $\Pic X$ on $\Pic X'$ and $\Bun_2X$ by tensoring in the sense that all maps in the above diagram are equivariant if we identify $\Pic X$ with $\lrquot{F^\times}{\AA_F^\times}{\cO_{\AA_F}^\times}$. Taking orbits under these compatible actions yields the commutative diagram
 $$ \xymatrix{\lrquot{E^\times\AA_F^\times}{\AA_E^\times}{\cO_{\AA_E}^\times}\ar[rr]^(.55){1:1}\ar[d]^{\Theta_E} &&\rquot{\Pic X'}{p^\ast\Pic X}\ar[d]^{p_\ast}\\
	 \lrquot{G_F}{G_{\AA_F}}{K}\ar[rr]^(.55){1:1} && {\PBun}_2X \;.\hspace{-3pt}} $$

 Since $f$ is right $K$-invariant, we may take the quotient of the domain of integration by $T_{\AA_F}\cap K$ from the right, which is the image of $\cO_{\AA_E}^\times$ in $G_{\AA_F}$. We obtain the assertion of the theorem for some still undetermined value of $c$. The value of $c$ is computed by plugging in a constant function for $f$.
\end{proof}

\begin{thm}
 \label{thmA}
 The dimension of the space of unramified toroidal functions is finite, bounded by
 $$ \dim \cV_\tor \ \leq \ \#\,\bigl(\,\PBun_2 X - \{c_D\}_{[D]\in\ECl X} \bigr) \;. $$
\end{thm}

\begin{proof}
 First remark that given the inequality in the theorem, finite-dimensionality follows since the right hand set is finite. Indeed, by Lemma \ref{many_effective_divisors},
 $$ \PBun_2 X - \{c_D\}_{[D]\in\ECl X} \ \subset \ \bigl\{v \in \PBun_2 X \mid \delta(v)\leq m_X \bigr\} $$
 since $m_X\geq g_X-1$, and the latter set is finite.

 We now proceed with the proof of the inequality. Let $f\in\cV_\tor$. We will show by induction on $d=\deg D$ that the value of $f$ at a vertex $c_D$ with $[D]\in\ECl X$ is uniquely determined by the values of $f$ at the elements of $\PBun_2 X - \{c_D\}_{[D]\in\ECl X}$. This will prove the theorem.

 By Theorem \ref{geometric_toroidal_cond}, the condition for $f$ to lie in $\cV_\tor$ reads
 $$ \sum_{[\cL]\in(\Pic X' \,/\, p^\ast\Pic X)}\hspace{-0,7cm} \Phi(f)([p_\ast\cL]) \ = \ 0\;,\quad\text{for all }\Phi\in\cH. $$
 
 If $d=0$, take $\Phi$ as the identity element in $\cH_K$. We know from Proposition \ref{delta_traces} that $p_\ast(\Pic X' \,/\, p^\ast\Pic X)=\PBuntr X \cup \{c_0\}$, so $f(c_0)$ equals a linear combination of values of $f$ at vertices $v$ in $\PBuntr X$, which all satisfy $\delta(v)<0$. Since the zero divisor class is the only class in $\ECl X$ of degree $0$, we have proven the case $d=0$.
 
 Next, let $D$ be an effective divisor of degree $d>0$ and put $\Phi=\Phi_D$. If $v$ is a $\Phi_D$-neighbour of $w$, then $\delta(v)$ and $\delta(w)$ can differ at most by $d$ (Lemma \ref{neighbours_of_c_0} \eqref{nb1}). Therefore all $\Phi_D$-neighbours $v$ of vertices in $\PBuntr X$ have $\delta(v)<d$. The vertex $c_D$ is the only $\Phi_D$-neighbour $v$ of $c_0$ with $\delta(v)=d$ (Lemma \ref{neighbours_of_c_0} \eqref{nb2}). Thus
 $$ 0 \quad  = \hspace{-0,4cm} \sum_{\cL\in(\Pic X' \,/\, p^\ast\Pic X)}\hspace{-0,7cm} \Phi_D(f)([p_\ast\cL]) \quad = \quad (q+1)\,f(c_D)\quad+\hspace{-0,6cm} \sum_{\begin{subarray}{c}\vspace{0,1cm}\cL\in(\Pic X' \,/\, p^\ast\Pic X),\\\vspace{0,1cm}([p_\ast\cL],v,\lambda)\in\cU_D([p_\ast\cL]),\\\delta(v)<d\end{subarray}} \hspace{-0,9cm} \lambda\ f(v) $$
 determines $f(c_D)$ as the linear combination of values of $f$ at vertices $v$ with $\delta(v)<d$. By the inductive hypothesis, $f(c_D)$ is already determined by the values of $f$ at vertices that are not contained in $\{c_D\}_{[D]\in\ECl X}$.
\end{proof}

\begin{eg}
 If $X$ is the projective line over $\FF_q$, then all vertices $v$ are of the form $c_D$ for some effective divisor $D$ (see Example \ref{eg_P1}). Thus $\cV_\tor$ is trivial. Since only $v=c_0$ satisfies $\delta(v)\leq m_X$, all values of $f\in\cV_0$ are multiples of $f(c_0)$. However, $\Ext^1(\cO_X,\cO_X)$ is trivial, thus the cuspidal condition (applied to the trivial Hecke operator) is $f(c_0)=0$. Thus also $\cV_0$ is trivial. See \cite{Lorscheid3} for the corresponding spaces in the case of an elliptic curve.
\end{eg}


\begin{appendix}
\section{Examples for rational function fields}
\label{appendix_examples}

\noindent
The appendix contains examples of graphs of Hecke operators for a rational function field, which can be calculated by elementary matrix manipulations. We do not exercise all calculations, but hint on how to do them. The reader will find examples for elliptic function fields that are determined by geometric methods in \cite{Lorscheid3}.

Let $F$ be $\FF_q(T)$, the function field of the projective line over $\FF_q$, which has $q+1$ $\FF_q$-rational points and trivial class group. Fix a place $x$ of degree $1$. 
 
\begin{pg} 
 Using strong approximation for $\SL_2$ (cf.\ Proposition \ref{strong_approx+comp}, where $J$ is trivial in this case), we see that the map obtained by adding the identity matrix $e$ at all places $y\neq x$,
 $$\begin{array}{ccc}
  \lrquot{\Gamma}{G_x}{K_x} & \longrightarrow & \lrquot{G_F}{G_\AA}{K} \;, \\
                 {[g_x]}       & \longmapsto     &             {[(g_x,e)]}
 \end{array}$$
 is a bijection. 
 
 We introduce some notation. Elements of $\cO^x_F=\bigcap_{y\neq x}(\cO_y\cap F)$ can be written in the form $\sum_{i=m}^0 b_i\pi_x^i$ with $b_i\in\FF_q$ for $i=m,\dotsc,0$ for some integer $m\leq0$. Let $\tilde K_x=\GL_2(\cO_x)$, where we view $\cO_x$ as the collection of all power series $\sum_{i\geq0} b_i\pi_x^i$ with $b_i\in\FF_q$ for $i\geq0$. Let $\Gamma=\GL_2(\cO^x_F)$ and let $Z$ be the center of $\GL_2$.
 \end{pg}

\begin{pg}
 \label{tricks}
 For better readability, we write $\pi$ for the uniformizer $\pi_x$ at $x$ and $g$ for a matrix in $G_x$. We say $g\sim g'$ if they represent the same class $[g]=[g']$ in $\lrquot{\Gamma}{G_x}{K_x}$, and indicate by subscripts to `$\sim$' how to alter one representative to another. The following changes of the representative $g$ of a class $[g]\in\lrquot{\Gamma}{G_x}{K_x}$ provide an algorithm to determine a standard representative for the class of any matrix $g\in G_x$:
 
\begin{enumerate}
 \item\label{trick1} By the Iwasawa decomposition, every class in $\lrquot{\Gamma}{G_x}{K_x}$ is represented by an upper triangular matrix, and 
                     $$ \smallmat a b {} d \ \ \ \putunder{\sim}{/\,Z_x} \ \ \ \smallmat a b {} d \ \ \ \smallmat {d^{-1}} {} {} {d^{-1}} \ \ \ = \ \ \ \smallmat a/d b/d {} 1 \;. $$
 \item\label{trick2} Write $a/d=r\pi^n$ for some integer $n$ and $r\in\cO_x^\times$, then with $b'=b/d$, we have
                     $$ \smallmat {r\pi^n} b' {} 1 \ \ \ \putunder{\sim}{/\,\tilde K_x} \ \ \ \smallmat {r\pi^n} b' {} 1 \ \ \ \smallmat {r^{-1}} {} {} 1 \ \ \ = \ \ \ \smallmat {\pi^n} b' {} 1 \;. $$
 \item\label{trick3} If $b'=\sum_{i\geq m} b_i \pi^i$ for some integer $m$ and coefficients $b_i\in\FF_q$ for $i\geq m$, then
  \begin{eqnarray*} 
    \smallmat {\pi^n} {\sum_{i\geq m} b_i \pi^i} {} 1 & \putunder{\sim}{/\,\tilde K_x} & \smallmat {\pi^n} {\sum_{i\geq m} b_i \pi^i} {} 1 \ \ \ \smallmat 1 {-\pi^{-n}(\sum_{i\geq n}b_i\pi^i)} {} 1 \\ 
     &=& \smallmat {\pi^n} {b_m \pi + \ldots + b_{n-1} \pi^{n-1}} {} 1 \;.
  \end{eqnarray*}
 \item\label{trick4} One can further perform the following step:
\begin{multline*}
    \smallmat {\pi^n} {b_m \pi^m + \ldots + b_{n-1} \pi^{n-1}} {} 1 \\
      \begin{array}{cl} \putunder{\sim}{\Gamma\,\backslash} & \smallmat 1 {-(b_m \pi^m + \ldots + b_0 \pi^0)} {} 1 \ \ \  \smallmat {\pi^n} {b_m \pi^m + \ldots + b_{n-1} \pi^{n-1}} {} 1         \vspace{5pt} \\
                         =& \smallmat {\pi^n} {b_1 \pi + \ldots + b_{n-1} \pi^{n-1}} {} 1 \;.  \end{array}
\end{multline*}
 \item\label{trick5} If $b=b_1 \pi + \ldots + b_{n-1} \pi^{n-1}\neq 0$, then $b=s\pi^k$ with $1\leq k \leq n-1$, $s \in O_x^\times$ and
  \begin{eqnarray*} 
    \smallmat \pi^n s\,\pi^{k} {} 1 \!\!\! & \putunder{\sim}{\Gamma\,\backslash \ /\,Z_x\,\tilde K_x} & \!\!\! \smallmat {} 1 1 {} \ \ \smallmat \pi^n s\,\pi^{k} {} 1 \ \ \smallmat {s^{-1}\pi^{-k}} {} {} {s^{-1}\pi^{-k}} \ \ \smallmat -s^2 {} s\pi^{n-k} 1 \\ 
    &=& \smallmat \pi^{n-2k} s^{-1}\pi^{-k} {} 1 \;.
  \end{eqnarray*} 
 \item\label{trick6} The last trick is
   $$ \smallmat \pi^{n} {} {} 1 \ \ \ \putunder{\sim}{\Gamma\,\backslash \ /\,Z_x\,\tilde K_x} \ \ \ \smallmat {} 1 1 {} \ \ \ \smallmat \pi^{n} {} {} 1 \ \ \ \smallmat \pi^{-n} {} {} \pi^{-n} \ \ \ \smallmat {} 1 1 {} \ \ \ = \ \ \ \smallmat \pi^{-n} {} {} 1 \;. $$
\end{enumerate}

 Executing these steps (possibly \eqref{trick3}--\eqref{trick5} several times) will finally lead to a matrix of the form 
 $$ p_n = \smallmat \pi^{-n} {} {} 1 $$
 for some $n\geq 0$. The matrix $p_n$ represents the vertex $c_{nx}$ in $\Vertex\cG_{\Phi,K}=\{c_{nx}\}_{n\geq0}$ where $\Phi$ is any unramified Hecke operator (cf.\ Example \ref{eg_P1}).Thus we found a way to determine the vertex $c_{nx}$ that is represented by an arbitrary matrix $g\in G_x\subset G_\AA$.
\end{pg} 

\begin{eg}[Graph of $0$ and $1$]
 According to paragraph \ref{algebra_strucure_in_graphs}, the graphs for the zero element $0$ and the identity $1$ in $\cH_K$ are as illustrated in Figures \ref{figure_graph_rat_0} and \ref{figure_graph_rat_1}, respectively.
\end{eg}

\begin{figure}[htb]
 \begin{center}
  \includegraphics{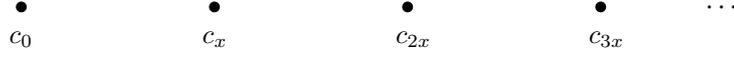}
  \caption{The graph of the zero element in $\cH_K$}
  \label{figure_graph_rat_0}
 \end{center} 
\end{figure}

\begin{figure}[htb]
 \begin{center}
  \includegraphics{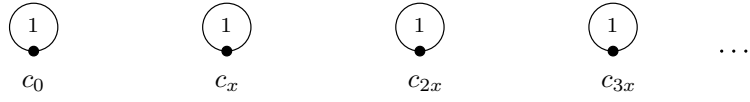}
  \caption{The graph of the identity in $\cH_K$}
  \label{figure_graph_rat_1}
 \end{center} 
\end{figure}

\begin{eg}[Graph of $\Phi_x$]
 \label{graph_rat_degree_1}
  By Proposition \ref{Phi_x_neighbours} the $\Phi_x$-neighbors of $p_i$ are of the form $p_i\xi_w$. With help of the reduction steps \eqref{trick1}--\eqref{trick6} in paragraph \ref{tricks} one can determine easily the standard representative $p_j$ of $p_i\xi_w$. We reobtain the graph of $\Phi_x$ as illustrated in Figure \ref{figure_graph_rat_deg_1} (cf.\ Example \ref{eg_P1}). 
\end{eg}

\begin{figure}[htb]
 \begin{center}
  \includegraphics{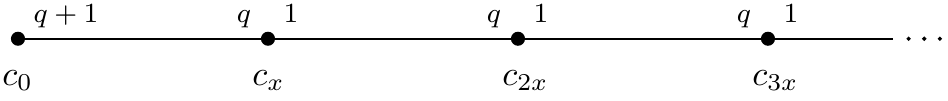}
  \caption{The graph of $\Phi_x$}
  \label{figure_graph_rat_deg_1}
 \end{center} 
\end{figure}

\begin{eg}[Graph of $\Phi_y$ for $y\neq x$]
 \label{graph_of_Phi_y_example}
 If we want to determine the edges of $\cG_y$ for a place $y$ of degree $d$ that differs from $x$, 
 we have to find the standard representative $p_j$ for elements 
 $$ p_i\ \ \smallmat {\pi_y} b {} 1 \hspace{0,5cm}\text{with }b\in\kappa_y,\text{ and}\hspace{0,5cm} p_i\ \ \smallmat 1 {} {} {\pi_y} \;. $$
 As $F$ has class number $1$, we can assume that $\pi_y\in F$ has nontrivial valuation in $y$ and $x$ only. 
 Let $\gamma\in G_F$ denote the inverse of one of the matrices $\tinymat {\pi_y} b {} 1 ,\tinymat 1 {} {} {\pi_y} $.
 For all places $z\neq x,y$, the canonical embedding $G_F\to G_z$ sends 
 $\gamma$ to a matrix $\gamma_z\in K_z$ since $v_z(\pi_y)=0$ by assumption.
 Thus multiplying with $\gamma\in G_F$ from the left, which operates diagonally on the components of all places,
 and multiplying componentwise with $\gamma_z^{-1}\in K_z$ from the right for all $z\neq x,y$,
 gives an element that is nontrivial only in $x$ (also compare with \cite[Lemma 3.7]{Gelbart1}).
 The matrices that we obtain in this way are:
 $$ \smallmat {\pi_x^d} b_0+\dotsb+b_{d-1}\pi_x^{d-1} {} 1 \ \ p_i\hspace{0,5cm}\text{with }b_i\in\kappa_x\text{ for }i=0,\dotsc,d-1,
    \text{ and}\hspace{0,5cm} \smallmat 1 {} {} {\pi_x^d} \ \ p_i \;. $$
 The reduction steps \eqref{trick1}--\eqref{trick6} of paragraph \ref{tricks} tell us which classes are represented,
 and we are able to determine the edges similarly to the previous example.
 Thus we obtain that $\cG_y$ only depends on the degree of $y$.
 Note that if $y$ is of degree $1$, then $\cG_y$ equals $\cG_x$.
 Figures \ref{figure_graph_rat_deg_2}, \ref{figure_graph_rat_deg_3}, \ref{figure_graph_rat_deg_4}, 
 and \ref{figure_graph_rat_deg_5} show the graphs for degrees $2$, $3$, $4$ and $5$, respectively.
\end{eg}

\begin{figure}[htb]
 \begin{center}
  \includegraphics{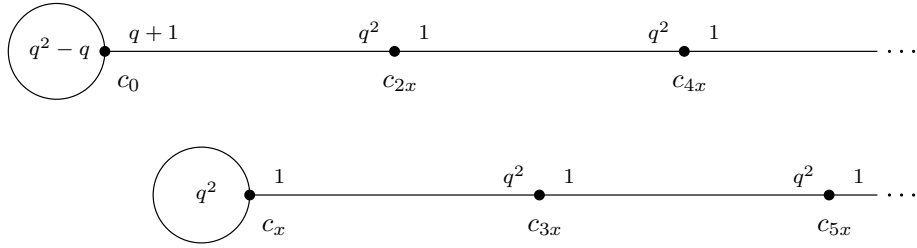}
  \caption{The graph of $\Phi_y$ for a place $y$ of degree $2$}
  \label{figure_graph_rat_deg_2}
 \end{center} 
\end{figure}

\begin{figure}[htb]
 \begin{center}
  \includegraphics{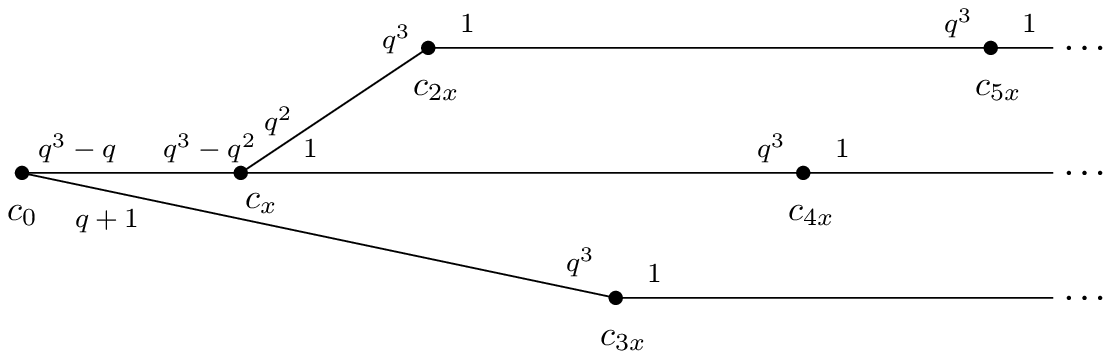}
  \caption{The graph of $\Phi_y$ for a place $y$ of degree $3$}
  \label{figure_graph_rat_deg_3}
 \end{center} 
\end{figure}

\begin{figure}[htb]
 \begin{center}
  \includegraphics{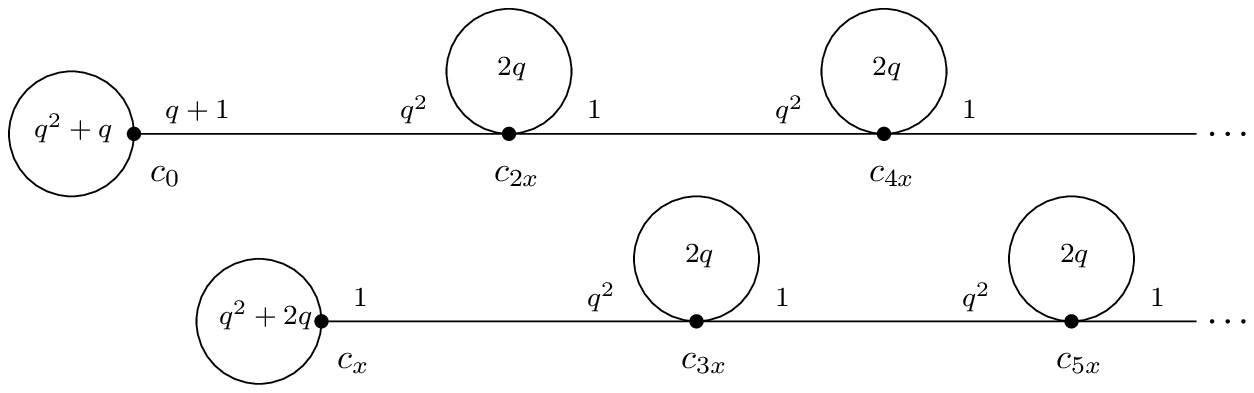}
  \caption{The graph of $\Phi_x^2$}
  \label{figure_graph_rat_deg_1_squared}
 \end{center} 
\end{figure}

\begin{figure}[htb]
 \begin{center}
  \includegraphics{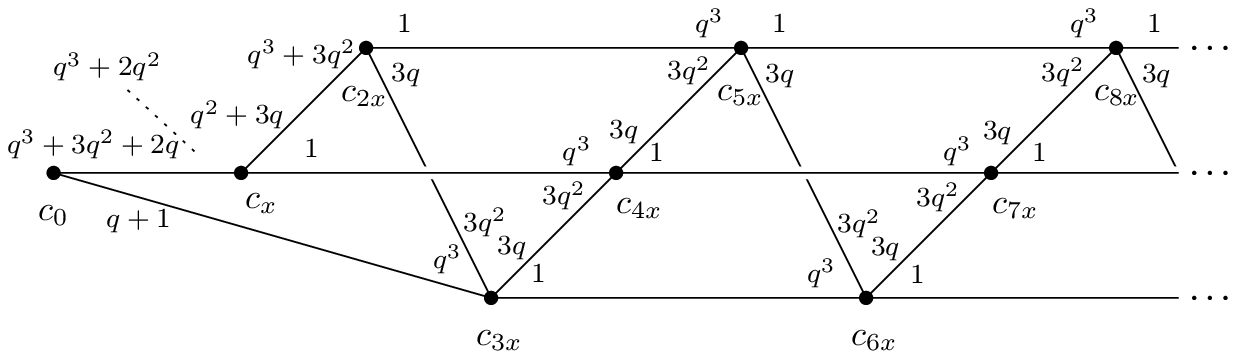}
  \caption{The graph of $\Phi_x^3$}
  \label{figure_graph_rat_deg_1_cubed}
  \ \vspace{-0,5cm}\\ \  
 \end{center} 
\end{figure}

\begin{figure}[htb]
 \begin{center}
  \includegraphics{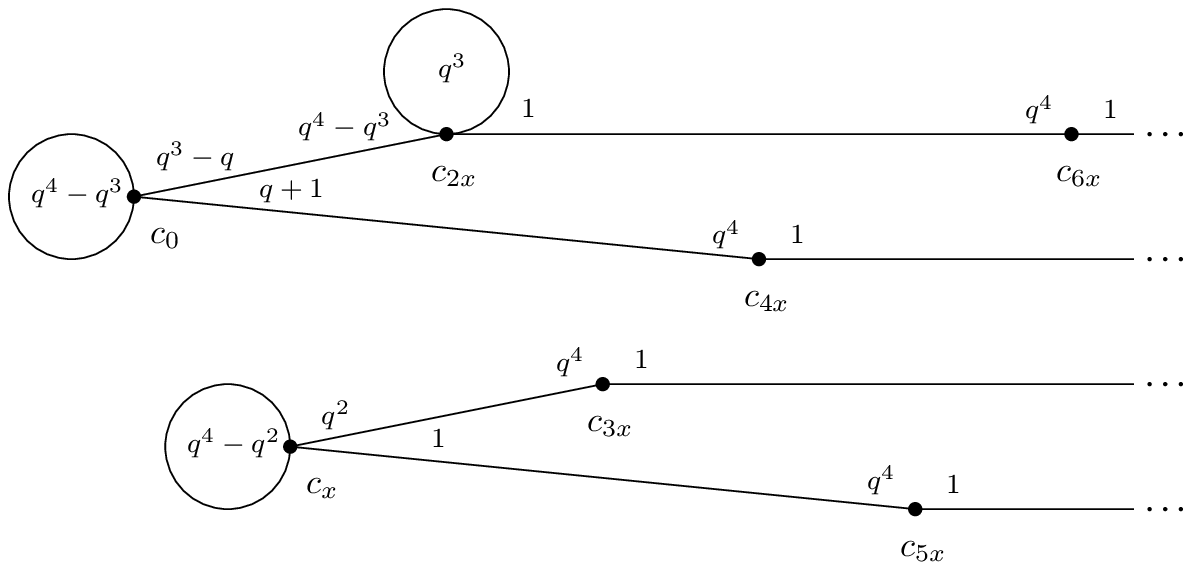}
  \caption{The graph of $\Phi_y$ for a place $y$ of degree $4$}
  \label{figure_graph_rat_deg_4}
 \end{center} 
\end{figure}

\begin{figure}[htb]
 \begin{center}
  \includegraphics{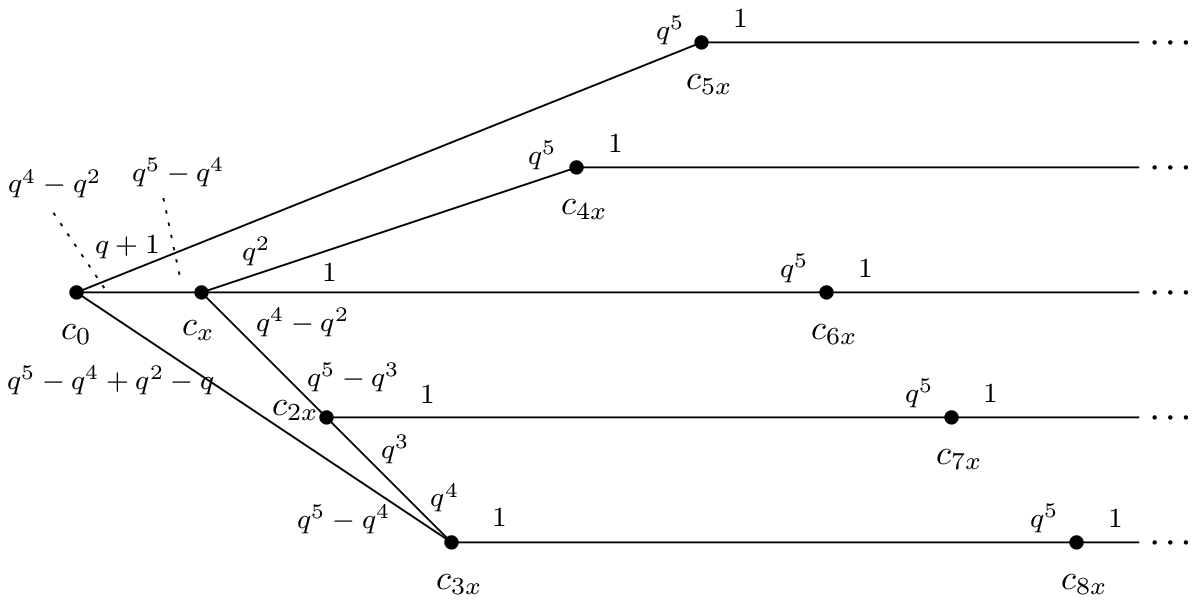}
  \caption{The graph of $\Phi_y$ for a place $y$ of degree $5$}
  \label{figure_graph_rat_deg_5}
 \end{center} 
\end{figure}

\begin{eg}[The graph of powers of $\Phi_x$]
 \label{powers_of_Phi_x_example}
 It is interesting to compare the graph of $\Phi_y$ with $\deg y=d$ to the graph of $\Phi_x^d$. 
 The latter graph is easily deduced from $\cG_x$ by means of paragraph \ref{algebra_strucure_in_graphs}.
 Namely, a vertex $v'$ is a $\Phi_x^d$-neighbour of a vertex $v$ in $\cG_{\Phi_x^d,K}$ if there is a path of
 length $d$ from $v$ to $v'$ in $\cG_x$, i.e.\ a sequence $(v_0,v_1,\dotsc,v_d)$ of vertices in $\cG_x$ with
 $v_0=v$ and $v_d=v'$ such that for all $i=1,\dotsc,d$, there is an edge $(v_{i-1}, v_i,m_i)$ in $\cG_x$. 
 The weight of an edge from $v$ to $v'$ in the graph of $\cG_x^d$ is obtained by taking the sum of the products $m_1\cdot\dotsc\cdot m_d$
 over all paths of length $d$ from $v$ to $v'$ in $\cG_x$.
 
 Figure \ref{figure_graph_rat_deg_1_squared} and \ref{figure_graph_rat_deg_1_cubed} 
 show the graphs of $\Phi_x^2$ and $\Phi_x^3$, respectively, and we see that for $\deg y=2$, we have $\Phi_x^2\equiv \Phi_y+2q\cdot 1\ (\textup{mod}\, \cJ(K))$ and for $\deg y=3$, we have $\Phi_x^3\equiv \Phi_y+3q\cdot \Phi_x\ (\textup{mod}\, \cJ(K))$ where $\cJ(K)$ is the ideal of $\cH_K$ of Hecke operators that operate trivial on $C^0(G_\AA)$.
\end{eg}

\begin{figure}[htb]
 \begin{center}
  \includegraphics{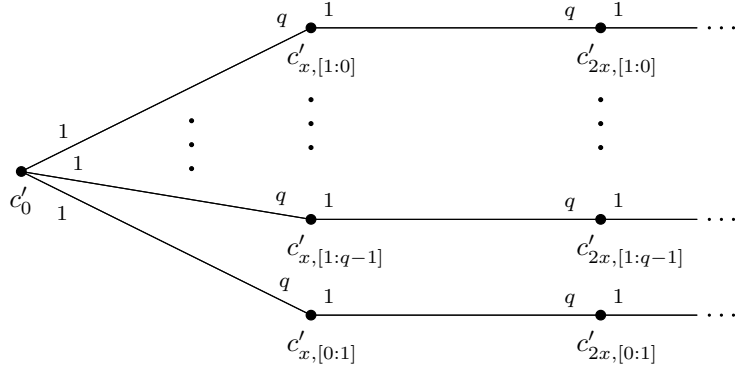}
  \caption{Graph of $\Phi'_{y,e}$ as defined in Example \ref{ramification_example}}
  \label{graph_ram_y}
 \end{center} 
\end{figure}

\begin{figure}[htb]
 \begin{center}
  \includegraphics{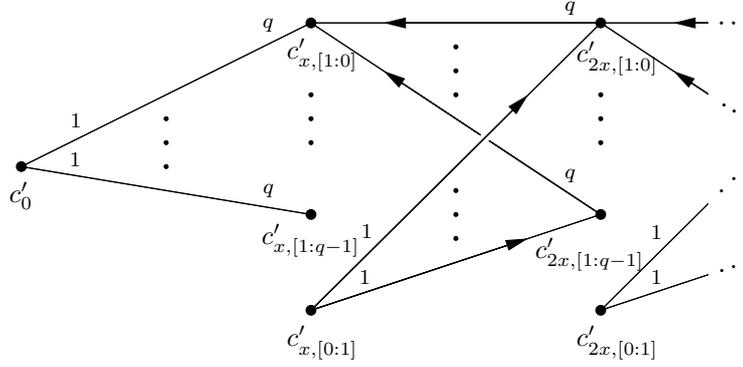}
  \caption{Graph of $\Phi'_x$ as defined in Example \ref{ramification_example}}
  \label{graph_ram_x}
 \end{center} 
\end{figure}

\begin{eg}[The graphs of two ramified Hecke operators]
 \label{ramification_example}
 It is also possible to determine examples for Hecke operators in $\cH_{K'}$ by elementary matrix manipulations, 
 when $K'<K$ is a subgroup of finite index. We will show two examples, 
 which are illustrated in Figures \ref{graph_ram_y} and \ref{graph_ram_x}.
 We omit the calculation, but only point out why the crucial differences between the two graphs occur.

 For $K'=\{k\in K\mid k_x\equiv\tinymat 1 {} {} 1 \pmod{\pi_x}\}$, the fibres of the projection
 $$ P: \lrquot{G_F}{G_\AA}{ K'} \longrightarrow \lrquot{G_F}{G_\AA}{ K} $$
 are given by $P^{-1}(c_0)=\{[p_0]\}$ and for positive $n$, by $P^{-1}(c_{nx})=\{[p_{nx}\vartheta_w]\}_{w\in\PP^1(\kappa_x)}$ with
 $\vartheta_{[1:c]}=\tinymat 1 c {} 1 $ and $\vartheta_{[0:1]}=\tinymat {} 1 1 {} $. 
 The union of these fibres equals the set of vertices of a Hecke operator in $\cH_{K'}$. We shall denote the vertices by
 $c'_0=[p_0]$ and $c'_{nx,w}=[p_{nx}\vartheta_w]$ for $n\geq1$ and $w\in\PP^1(\kappa_x)$.
 Note that $G_{\FF_q}=G_{\kappa_x}$ acts on $\PP^1(\kappa_x)$ from the right, 
 so if $\gamma\in G_{\FF_q}$, then $w\mapsto w\gamma$ permutes the elements of $\PP^1(\kappa_x)$.

 The first Hecke operator $\Phi'_{y,\gamma}\in\cH_{K'}$ that we consider 
 is $(\vol K/\vol K')$ times the characteristic function of $K'\tinymat {\pi_y} {} {} 1 \gamma K'$,
 where $y$ is a degree one place different to $x$ and $\gamma\in G_\AA$
 is a matrix whose only nontrivial component is $\gamma_x\in G_{\FF_q}$.
 (The factor $(\vol K/\vol K')$ is included to obtain integer weights).
 Since $K'\tinymat {\pi_y} {} {} 1 \gamma K'\subset K\tinymat {\pi_y} {} {} 1 \gamma K$, 
 the graph of $\Phi'_{y,\gamma}$ relative to $K'$ can have an edge from $v$ to $w$ only if $\cG_y$ has an edge from $P(v)$ to $P(w)$.
 Because $K'_y=K_y$, we argue as for $K$ that 
 $K'\tinymat {\pi_y} {} {} 1 \gamma K'=\coprod_{w\in\PP^1(\kappa_y)}\xi_w \gamma K'$.
 Applying the same methods as in Example \ref{graph_of_Phi_y_example}, one obtains that 
 $$ \cU_{\Phi'_{y,\gamma},K'}(c'_0) \ = \ \{(c'_0,c'_{x,w},1)\}_{w\in\PP^1(\kappa_x)} $$
 and for every $n\geq1$ and $w\in\PP^1(\kappa_x)$ that 
 $$ \cU_{\Phi'_{y,\gamma},K'}(c'_{nx,w}) \ = \ \{(c'_{nx,w},c'_{(n+1)x,w\gamma},1),(c'_{nx,w},c'_{(n-1)x,w\gamma},q)\} \;. $$
 For the case that $\gamma$ equals the identity matrix $e$, the graph is illustrated in Figure \ref{graph_ram_y}.
 Note that for general $\gamma$, an edge does not necessarily have an inverse edge since $w\gamma^2$ does not have to equal $w$.
 
 The second Hecke operator $\Phi'_x\in\cH_{K'}$ is $(\vol K/\vol K')$ times the characteristic function of $K'\tinymat {\pi_x} {} {} 1 K'$. This case behaves differently, since $K'_x$ and $K_x$ are not equal; in particular, we have $K'\tinymat {\pi_x} {} {} 1 K' = \coprod_{b\in\kappa_x}\tinymat {\pi_x} {b\pi_x} {} 1 K'$. This allows us to compute the edges as illustrated in Figure \ref{graph_ram_x}. Note that for $n\geq1$, the vertices of the form $c'_{nx,[1:0]}$ and $c'_{nx,[0:1]}$ behave particularly.
\end{eg}

\end{appendix}

\bigskip 

\bibliographystyle{plain}

\end{document}